\documentclass[11pt]{amsart}
\usepackage{amsmath}
\usepackage{amsfonts}
\usepackage{amssymb}
\usepackage{amsthm}
\usepackage{amscd}
\usepackage{ifpdf}
\ifpdf  
\usepackage{hyperref}           
\else
\usepackage[hypertex]{hyperref}
\fi
\usepackage{bookmark}
\usepackage{cite}

\allowdisplaybreaks[1]
\newtheorem{theorem}{Theorem}[section]
\newtheorem{lemma}[theorem]{Lemma}
\newtheorem{proposition}[theorem]{Proposition}
\newtheorem{corollary}[theorem]{Corollary}

\theoremstyle{definition}
\newtheorem{definition}[theorem]{Definition}

\theoremstyle{remark}
\newtheorem{remark}[theorem]{Remark}

\renewcommand\urladdr[1]{}  

\begin{document}
\title[Classification of operator systems and operator spaces]{The
classification problem for finitely generated operator systems and spaces}
\author[Argerami]{Martin Argerami}
\address{Martin Argerami\\
Department of Mathematics and Statistics\\
University of Regina\\
Regina SK S4S 0A2, Canada}
\email{argerami@uregina.ca}
\urladdr{ev03023.math.uregina.ca}
\author[Coskey]{Samuel Coskey}
\address{Samuel Coskey\\
Department of Mathematics\\
Boise State University\\
1910 University Drive\\
Boise ID 83725-1555, USA}
\email{scoskey@nylogic.org}
\urladdr{boolesrings.org/scoskey}
\author[Kalantar]{Mehrdad Kalantar}
\address{Mehrdad Kalantar\\
School of Mathematics and Statistics\\
Carleton University\\
4302 Herzberg Laboratories\\
Ottawa ON K1S 5B6 Canada}
\email{mkalanta@math.carleton.ca}
\urladdr{people.math.carleton.ca/~mkalanta/Home.html}
\author[Kennedy]{Matthew Kennedy}
\address{Matthew Kennedy\\
School of Mathematics and Statistics\\
Carleton University\\
4302 Herzberg Laboratories\\
Ottawa ON K1S 5B6, Canada}
\email{mkennedy@math.carleton.ca}
\urladdr{people.math.carleton.ca/~mkennedy}
\author[Lupini]{Martino Lupini}
\address{Martino Lupini\\
Department of Mathematics and Statistics\\
N520 Ross, 4700 Keele Street\\
Toronto Ontario M3J 1P3, Canada, and Fields Institute for Research in
Mathematical Sciences\\
222 College Street\\
Toronto ON M5T 3J1, Canada}
\email{mlupini@mathstat.yorku.ca}
\urladdr{www.lupini.org}
\author[Sabok]{Marcin Sabok}
\address{Marcin Sabok\\
McGill University\\
Department of Mathematics and Statistics\\
Burnside Hall, Room 1005\\
805 Sherbrooke Street West\\
Montreal QC H3A 0B9, Canada}
\email{marcin.sabok@mcgill.ca}
\thanks{Martino Lupini was supported by the York University Susan Mann
Dissertation Scholarship. This work was completed during the Focussed
Research Group program ``Borel complexity and classification of operator
systems'' at the Banff International Research Station. The hospitality of
BIRS is gratefully acknowledged.}
\dedicatory{}

\begin{abstract}
The classification of separable operator systems and spaces is commonly
believed to be intractable. We analyze this belief from the point of view of
Borel complexity theory. On one hand we confirm that the classification
problems for arbitrary separable operator systems and spaces are
intractable. On the other hand we show that the finitely generated operator
systems and spaces are completely classifiable (or smooth); in fact a
finitely generated operator system is classified by its complete theory when
regarded as a structure in continuous logic. In the particular case of
operator systems generated by a single unitary, a complete invariant is
given by the spectrum of the unitary up to a rigid motion of the circle,
provided that the spectrum contains at least $5$ points. As a consequence of
these results we show that the relation on compact subsets of $\mathbb{C}^{n}
$, given by homeomorphism via a degree $1$ map, is smooth.
\end{abstract}

\maketitle
\tableofcontents

\section{Introduction}

An operator system is a closed unital self-adjoint subspace of a unital
C*-algebra. The study of operator systems can be traced back to the pioneering work of Arveson in the late 1960s and early 1970s \cite{arveson_subalgebras_1969,arveson_subalgebras_1972}.
The foundational result of the theory is the Choi--Effros abstract
characterization in terms of positive cones on matrix amplifications \cite{choi_injectivity_1977}; see also \cite[Theorem 13.1]{paulsen_completely_2002}. Since then the theory of operator systems has
expanded substantially thanks to the work of Blecher, Paulsen, and many
others \cite{paulsen_completely_2002}.

Operator systems offer a natural framework to study complete positivity,
which is an essential tool in operator algebras, as well as a number of
problems in quantum information theory. For instance the Tsirelson problem,
teleportation, bounded entanglement, and superdense coding all admit
equivalent reformulations in terms of operator systems \cite{bodmann_frame_2007, bodmann_smooth_2007, bodmann_decoherence_2007,
johnston_computing_2009, johnston_minimal_2011}. Moreover Farenick, Kavruk,
Paulsen, and Todorov found in \cite{farenick_operator_2012,kavruk_nuclearity_2011,farenick_characterisations_2013}
reformulations of the Connes embedding problem in terms of tensor products
of finite-dimensional operator systems. Of course, the Connes embedding
problem is one of the most important open problems in operator algebras.

In the last few years several great advances have occurred in the theory of
operator systems. By work of Dritschell--McCullough, Arveson,
Davidson--Kennedy and others \cite{dritschel_boundary_2005,
arveson_noncommutative_2008, davidson_choquet_2013} the existence of
sufficiently many boundary representations to generate the C*-envelope has
been established, confirming Arveson's intuition from almost 50 years
earlier. (The existence of the C*-envelope had been previously established
by Hamana via his theory of injective envelopes \cite{hamana_injective_1979,hamana_injective_1979-1}.) Despite these recent
advances, classification results for operator systems are very rare. The
only more or less explicit classification result to this date is Arveson's
classification of operator systems acting on a finite-dimensional Hilbert
space from \cite{arveson_noncommutative_2010}.

We recall the definition of unital completely positive map, which is the
natural notion of morphism between operator systems. If $\phi\colon
X\rightarrow Y$ is a linear map between operator systems, then the $n^{\rm th}$ 
\emph{amplification} $id_{M_{n}}(\mathbb{C})\otimes \phi $ of $\phi $ is the linear map
from $M_{n}(\mathbb{C})\otimes X$ to $M_{n}(\mathbb{C})\otimes Y$ given by 
\begin{equation*}
\sum_{i}\alpha _{i}\otimes x_{i}\mapsto \sum_{i}\alpha _{i}\otimes \phi
\left( x_{i}\right)\text{.}
\end{equation*}
(Here, $\otimes$ denotes the algebraic tensor product.) Under the canonical
identification of $M_{n}(\mathbb{C})\otimes X$ and $M_{n}(\mathbb{C})\otimes Y$ with $M_{n}(X)$ and $
M_{n}(Y)$ this map has the form
\begin{equation*}
\left[ x_{ij}\right] \mapsto \left[ \phi (x_{ij})\right]
\end{equation*}
and it is often denoted by $\phi^{(n)}$. A linear map $\phi\colon
X\rightarrow Y$ is \emph{unital completely positive} if $\phi \left(
1\right) =1$ and for every $n$ the amplification $\phi^{(n)}=id_{M_{n}(\mathbb{C})}
\otimes \phi $ is positive, i.e., it sends positive elements in $M_n(X)$ to
positive elements in $M_n(Y)$.

The operator systems together with the unital completely positive maps form
a category, and isomorphism in this category is called \emph{complete order
isomorphism}. It is worth recalling that a unital map between operator
systems is completely positive if and only if it is completely contractive 
\cite[Proposition 3.6]{paulsen_completely_2002}.
As a consequence, complete order isomorphisms
are precisely the unital invertible complete isometries. Moreover a complete
order isomorphism between C*-algebras is automatically multiplicative \cite[
Corollary 1.3.10]{blecher_operator_2004}, and hence a *-isomorphism. In the
following, the classification of operator systems will always mean
classification up to complete order isomorphism.

It seems to be commonly believed that any meaningful classification of even
finitely generated operator systems is out of reach. In this paper we show
that, while on one hand the classification problem for all separable
operator systems is intractable due to a result of Sabok \cite
{sabok_completeness_2013}, on the other hand the classification problem for
finitely generated operator systems is tractable. More precisely we prove
that from the point of view of Borel complexity theory, the classification
of finitely generated operator systems is as low in complexity as it can
conceivably be.

Borel complexity theory is an area of mathematics that studies the relative
complexity of classification problems using tools and methods of descriptive
set theory. In this framework our notions of complexity and classifiability
are made precise. Moreover powerful tools and criteria---such as Hjorth's
theory of turbulence \cite{hjorth_classification_2000}---have been developed
to rigorously exclude the possibility of certain classification results.

In the setting of Borel complexity theory, a classification problem is
regarded as an equivalence relation on a standard Borel space. This covers,
perhaps after a suitable parametrization, most concrete classification
problems in mathematics. The fundamental notion of comparison between
equivalence relations is the following: if $E,F$ are equivalence relations
on standard Borel spaces $X,Y$, then $E$ is \emph{Borel reducible} to $F$ if
there exists a Borel function $f\colon X\rightarrow Y$ satisfying 
\begin{equation*}
x\mathrel{E}x^{\prime }\quad \iff \quad f(x)\mathrel{F}f(x^{\prime })\text{.}
\end{equation*}

Borel complexity theory provides a number of \emph{benchmarks} of complexity
to obtain a hierarchy of classification problems in mathematics. For example
an equivalence relation is \emph{smooth} or \emph{concretely classifiable}
if it is Borel reducible to the relation of {equality} on the set of real
numbers (or any Polish space). Classification results of this type are the
most satisfactory. For instance the structure theorem for countable
divisible abelian groups shows that relation of isomorphism for such groups
is smooth. Similarly the relation of isomorphism of UHF C*-algebras is
smooth \cite{glimm_certain_1960}. But such results are quite strong and
therefore rare. Indeed the classification of torsion-free abelian groups of
rank $1$ is already nonsmooth.

It is therefore natural to extend the notion of classifiability and to allow
invariants that are higher in the complexity hierarchy. Consider the
relation $\cong _{\mathcal{C}}$ of isomorphism within some class $\mathcal{C}
$ of countable structures. A relation is \emph{classifiable by countable
structures} if it is Borel reducible to $\cong _{\mathcal{C}}$ for some
class of countable structures $\mathcal{C}$. Most concrete classification
results in mathematics can be viewed as a classification by countable
structures. For instance, in Elliott's celebrated classification of
separable AF algebras from \cite{elliott_classification_1976}, the
invariants include countable ordered abelian groups with a distinguished
element (the ordered $K_{0}$ groups).

Although classifiability by countable structures is a very general notion,
there are many very complex classification problems that are not
classifiable by countable structures. For examples, Sasyk--T\"{o}rnquist and
Farah--Toms--T\"{o}rnquist have shown in \cite{sasyk_classification_2009,
farah_turbulence_2014} using Hjorth's theory of turbulence that separable II$
_{1}$ factors and separable simple nuclear {C}*-algebras are \emph{not
classifiable by countable structures}.

The next natural benchmark of complexity is given by orbit equivalence
relations of Polish group actions. A relation is then \emph{classifiable by
orbits }if it is Borel reducible to the orbit equivalence relation
associated with the continuous action of a Polish group on a Polish space.
It was shown in \cite{elliott_isomorphism_2013} that separable operator
systems are classifiable by orbits. In fact a result of Sabok from \cite
{sabok_completeness_2013} shows that separable operator systems (and in fact
even simple, separable, nuclear C*-algebras) have \emph{maximal complexity }
among classes that are classifiable by orbits.

In this paper we show that the situation is very different for finitely
generated operator systems. Namely the classification problem for finitely
generated operator systems is smooth. We prove this result by showing that
the complete order isomorphism classes of finitely generated operator
systems are naturally parametrized by the points of a Polish space.

We give a second proof of the same fact using the framework of logic for
metric structures. The logic for metric structures, or continuous logic, is
a generalization of the usual first order logic that is suitable for
application to functional-analytic structures such as operator systems and
C*-algebras. In this framework operator systems are regarded as structures
in a suitable language $\mathcal{L}_{OSy}$ \cite[Appendix B]
{goldbring_kirchberg_2014}; see also \cite[Section 3.3]
{elliott_isomorphism_2013}. One can then assign to every operator system its
theory as an $\mathcal{L}_{OSy}$-structure, which is the collection of
evaluations of $\mathcal{L}_{OSy}$-formulas. For finitely generated operator
systems this provides a concrete---albeit hard to calculate---smooth
complete invariant. The same result applies to all collections of structures
in some language for which the domain of quantification is compact, such as
finite-dimensional operator spaces up to complete isometry.

In the even more special case of operator systems generated by a unitary we
provide a more explicit complete invariant. For unitaries with three or less
points in their spectra, isomorphism of the operator systems they generate
is simply determined by the cardinality of the spectrum. In the case of
unitaries with five or more points in their spectrum, two operator systems
generated by unitaries are complete order isomorphic if and only if the
spectra of the generating unitaries are conjugate by a rigid motion of the
circle.

As a consequence of the smoothness result for finitely generated operator
systems, we draw similar conclusions for a natural relation of degree $1$
homeomorphism for compact subsets of $\mathbb{C}$ or, more generally, $
\mathbb{C}^{n}$. Here, compact subsets of $\mathbb{C}^n$ are said to be 
\emph{degree $1$ homeomorphic} if they are homeomorphic via a linear combination of 
1, $z$, $\bar{z}$, and $\bar{z}z$ (we call this a \emph{degree $1$ map}). For comparison, note that the
classification of compact subsets of $\mathbb{C}$ up to arbitrary
homeomorphism is not classifiable by countable structures. This latter
result is due to Farah--Toms--T\"{o}rnquist and was obtained using the
methods of \cite{hjorth_classification_2000}.

We also consider the classification problem for separable (unital) operator
spaces. An operator space is a linear subspace of a C*-algebra, while a
unital operator space is a unital linear subspace of a unital C*-algebra.
Similarly to operator systems, (unital) operator spaces admit abstract
characterizations; see \cite[Theorem 3.1]{ruan_subspaces_1988} and \cite[
Theorem 1.1]{blecher_metric_2011}. There are two natural relations of
equivalence for (unital) operator spaces: (unital) complete isometry and
(unital) completely bounded isomorphism. We show that the relation of
complete isomorphism of separable operator spaces has maximal complexity
among analytic equivalence relations. This follows directly from the
analogous result for Banach spaces \cite{ferenczi_complexity_2009}. Moreover
the relations of (unital) complete isometry of finitely generated (unital)
operator spaces is smooth. (It should be noted that all $N$-dimensional
operator spaces are completely isomorphic by \cite[Corollary 7.7]
{pisier_introduction_2003}.)

In addition to establishing the results described above, in the present
paper we take the broader aim of laying a foundation for the study of the
complexity of classification problems for operator systems and spaces. (The
initial results in this direction were obtained in \cite
{elliott_isomorphism_2013}.) To this purpose we consider many natural
parametrizations of operator spaces and systems, and then we show that they
are all (weakly) equivalent in the sense of \cite{farah_turbulence_2014}.
(The analogous results for parametrizations of C*-algebras have been
obtained in \cite{farah_turbulence_2014}.) As a consequence this implies
that any of these natural parametrizations can be used to assess the
complexity of some class of operator systems or spaces without affecting the
conclusions.

The rest of the paper is organized as follows. In Section \ref
{Section:parametrizing} we introduce many natural parametrizations of
operator systems and spaces. The full proof of the equivalence of these
parametrizations is given in the Appendices \ref{Appendix:equivalenceOSy}
and \ref{Appendix:equivalenceOSp}. In Section \ref{Section:separable} we
consider the classification problem for arbitrary separable operator systems
and spaces. In Section \ref{Section:fg} we specialize the analysis to
finitely generated operator systems, and show that they are concretely
classifiable. Finally in Section \ref{Section:structures} we present and
prove the more general smoothness result for $\mathcal{L}$-structures in
continuous logic.

In the following all the structures (Banach spaces, operator spaces,
operator systems, and $\mathcal{L}$-structures) are assumed to be separable,
complete, and nonzero. As usual in model theory we denote tuples of elements
by $\bar{x}$ and $\bar{y}$. We reserve the letter $z$ for a complex
variable, in which case $\bar{z}$ will denote the complex conjugate of $z$.
If $X$ is a set and $n\in \mathbb{N}$ we denote by $M_{n}(X)$ the set of $
n\times n$ matrices with entries from $X$. If $K$ is a field and $X$ is a $K$
-vector space, then $M_{n}(X)$ will be identified with the $K$-vector space $
M_{n}\left( K\right) \otimes X$.

The authors would like to thank David Blecher for referring them to the
articles  \cite{blecher_metric_2011,blecher_metric_2013}.

\section{Parametrizing operator systems and operator spaces\label
{Section:parametrizing}}

We consider in this section several natural standard Borel parametrizations
of the categories $\mathbf{OSy}$ and $\mathbf{OSp}$ of complete separable
operator systems and spaces. In Appendix \ref{Appendix:equivalenceOSy} and
Appendix \ref{Appendix:equivalenceOSp} we will show that all these
parametrizations are weakly equivalent.

\subsection{Standard Borel parametrizations}

Following \cite[Definition 2.1]{farah_turbulence_2014} a \emph{standard
Borel parametrization} of a category $\mathcal{C}$ is a pair $\left(
X,f\right) $ where $X$ is a standard Borel space and $f$ is a map from $X$
to the class of objects of $\mathcal{C}$, such that the range of $f$
contains an isomorphic copy of every object of $\mathcal{C}$. Two
parametrizations $\left( X,f\right) $ and $\left( Y,g\right) $ are called 
\emph{weakly equivalent} \cite[Definition 2.1]{farah_turbulence_2014} if
there are Borel functions $a:X\rightarrow Y$ and $b:Y\rightarrow X$ such
that $\left( g\circ a\right) (x)\cong f(x)$ and $\left( f\circ b\right)
(y)\cong g(y)$ for every $x\in X$ and $y\in Y$. If moreover one can choose $
a $ and $b$ to be injective, then the parametrizations $\left( X,f\right) $
and $\left( Y,g\right) $ are called \emph{equivalent}. As observed in \cite[
Section 2]{farah_turbulence_2014} it follows from the Borel version of the
Schr\"{o}der--Bernstein theorem \cite{kechris_classical_1995} that the
parametrizations $\left( X,f\right) $ and $\left( Y,g\right) $ are
equivalent if and only if there is a Borel isomorphism $\varphi $ from $X$
to $Y$ such that $\left( g\circ \varphi \right) (x)\cong f(x)$ and $\left(
f\circ \varphi ^{-1}\right) (y)\cong g(y)$ for every $x\in X$ and $y\in Y$.

Suppose that $\mathcal{C}$ is a category, and $\left( X,f\right) $ is a
parametrization of $\mathcal{C}$. A subcategory $\mathcal{C}_{0}$ of $
\mathcal{C}$ is \emph{Borel} (in the parametrization $f$) if the set
\begin{equation*}
X_{0}=\left\{ x\in X:f(x)\text{ is an object of }\mathcal{C}_{0}\right\}
\end{equation*}
is a Borel subset of $X$. The relation of isomorphism of $\mathcal{C}_{0}$
in the parametrization $f$ is the relation $E_{\cong }^{\mathcal{C}_{0}}$ on 
$X_{0}$ defined by
\begin{equation*}
x\mathrel{E}_{\cong }^{\mathcal{C}_{0}}{}x^{\prime }\text{ }\Leftrightarrow 
\text{ }f(x)\cong f(x^{\prime })\text{.}
\end{equation*}
It is clear that replacing $\left( X,f\right) $ with a weakly equivalent
parametrization $\left( X^{\prime },f^{\prime }\right) $ does not change the
notion of Borel subcategory. Moreover the isomorphism relation corresponding
to the parametrization $\left( X^{\prime },f^{\prime }\right) $ is Borel
bireducible with the isomorphism relation corresponding to the
parametrization $\left( X,f\right) $.

\subsection{Parametrizations of operator systems\label{Subsection:
parametrizations os}}

A $\mathbb{Q}(i)$\emph{-}$\ast $\emph{-vector} space $S$ is a $\mathbb{Q}(i)$
-vector space endowed with a conjugate linear involution $x\mapsto x^{\ast }$
. Denote by $\mathcal{V}$ the (countable) $\mathbb{Q}(i)$-$\ast $-vector
space of *-polynomials of degree at most $1$, with constant term, with
coefficients in $\mathbb{Q}(i)$, in the noncommutative variables $X_{n}$ for 
$n\in \mathbb{N}$. Similarly define $\mathcal{V}_{\mathbb{C}}$ the complex $
\ast $-vector space of *-polynomials of degree at most $1$, with constant
term, with complex coefficients, in the noncommutative variables $X_{j}$ for 
$j\in \mathbb{N}$. Suppose that $\left\{ \mathfrak{p}_{n}:n\in \mathbb{N}
\right\} $ is an enumeration of $\mathcal{V}$ such that $\mathfrak{p}_{1}$
is the constant polynomial $1$.

\subsubsection{The space $\Gamma $}

Let $H$ be the separable infinite-dimensional Hilbert space. Denote
by $B(H)$ the algebra of bounded linear operators on $H$. For every $n\in 
\mathbb{N}$ endow the $n$-ball $B_{n}(H)$ of $B(H)$ with the (compact
metrizable) weak operator topology. Finally endow $B(H)$ with the
corresponding inductive limit (standard) Borel structure, obtained by
setting $A\subset B(H)$ Borel iff $A\cap B_{n}(H)$ is Borel for every $n\in 
\mathbb{N}$. Denote by $\Gamma $ the set $B(H)^{\mathbb{N}}$ of sequences in 
$B(H)$ endowed with the product Borel structure. This can be seen as a
standard Borel parametrization of $\mathbf{OSy}$. For $\gamma \in \Gamma $
and $p\in \mathcal{V}_{\mathbb{C}}$ define $p(\gamma )$ to be the element of 
$B(H)$ obtained by replacing the variable $X_{i}$ with $\gamma _{i}$ and
interpreting a constant $c$ as the corresponding multiple $cI$ of the
identity operator. Every element $\gamma =\left( \gamma _{n}\right) _{n\in 
\mathbb{N}}$ of $\Gamma $ codes the separable operator system $\mathcal{OS}
y(\gamma )$ obtained by taking the closure in the norm topology of the set $
\left\{ p(\gamma ):p\in \mathcal{V}_{\mathbb{C}}\right\} $.

The space $\Gamma _{\mathcal{V}}$ of unital linear self-adjoint functions
from $\mathcal{V}$ to $B(H)$ is a Borel subset of $B(H)^{\mathcal{V}}$
endowed with the product Borel structure. This is also a standard Borel
parametrization of $\mathbf{OSy}$. An element $\varphi $ of 
${\Gamma }_{\mathcal{V}}$ codes the operator system $\mathcal{OS}y\left(
\varphi \right) $ which is the closure of the range of $\varphi $. The
function $\gamma \mapsto \varphi _{\gamma }$ from $\Gamma $ to $\Gamma _{
\mathcal{V}}$ defined by $\varphi _{\gamma }\left( p\right) =p(\gamma )$ is
a Borel isomorphism witnessing that the parametrizations $\Gamma $ to $
\Gamma _{\mathcal{V}}$ are equivalent.

\subsubsection{The space $\Xi $}

Denote by $\Xi $ the space of $\delta=\left( \delta _{n}\right) _{n\in 
\mathbb{N}}$ where $\delta_n\in \mathbb{R}^{M_{n}(\mathcal{V})}$ is such
that for some operator system $X$ and some nonzero dense sequence $\gamma
\in X$,
\begin{equation*}
\delta _{n}\left( \left[ p_{ij}\right] \right) =\left\Vert \left[
p_{ij}(\gamma )\right] \right\Vert _{M_{n}(X)}\text{.}
\end{equation*}
The ultraproduct construction shows that $\Xi $ is a Borel set. The operator
system $\mathcal{OS}y(\delta )$ associated with $\delta $, which is
completely isometric to $X$ as above, can be described as the Hausdorff
completion of $\mathcal{V}$ with respect to the seminorm $p\mapsto \delta
_{1}\left( p\right) $.

\subsubsection{The space $\widehat{\Xi }$}

Denote by $M_{n}(S)$ the $\mathbb{Q}(i)$-$\ast $-vector space of $n\times n$
matrices over $S$. A \emph{matrix order }on $S$ is a collection $\left(
C_{n}\right) _{n\in \mathbb{N}}$ of cones $C_{n}$ on $M_{n}(S)$ such that

\begin{enumerate}
\item $C_{n}\cap \left( -C_{n}\right) =\left\{ 0\right\} $, and

\item for every $n,m\in \mathbb{N}$ and every $n\times m$ matrix $A$ with
coefficients in $\mathbb{Q}(i)$, $A^{\ast }C_{n}A\subset C_{m}$.
\end{enumerate}

We call a selfadjoint $e\in S$ an \emph{order unit} if for every selfadjoint $x\in S$ there is 
$r\in \mathbb{Q}_{+}$ such that $re+x\in C_{1}$. An order unit is \emph{
Archimedean }if $re+x\in C_{1}$ for all $r\in \mathbb{Q}_{+}$ implies $x\in
C_{1}$. We call $e$ an Archimedean matrix order unit provided that $
I_{n}\otimes e\in M_{n}(S)$ is an Archimedean order unit for $M_{n}(S)$.

Suppose  $S$ is a matrix ordered $\mathbb{Q}(i)$-$\ast $-vector space
with an Archimedean matrix order unit. The same argument as \cite[page 176]
{paulsen_completely_2002} shows that $C_{n}$ is a full cone for every $n\in 
\mathbb{N}$, i.e.\ $C_{n}-C_{n}=M_{n}(S)$. Moreover the proof of \cite[
Proposition 13.3]{paulsen_completely_2002} yields that
\begin{equation*}
\left\Vert x\right\Vert =\inf \left\{ r:
\begin{bmatrix}
rI_{n} & x \\ 
x & rI_{n}
\end{bmatrix}
\in C_{2n}\right\}
\end{equation*}
is a norm on $M_{n}(S)$, and $C_{n}$ is a closed subset of $M_{n}(S)$ in the
topology induced by such norm.

The completion $\widehat{S}$ of $S$ with respect to such norm is then a 
\emph{complex} $\ast $-vector space. Moreover the closure $\widehat{C}_{n}$
of $C_{n}$ inside $M_{n}(\widehat{S})$ for $n\in \mathbb{N}$ form a matrix
order on $\widehat{S}$ with Archimedean matrix order unit $e$ . Therefore by
the abstract characterization of operator systems due to Choi and Effros 
\cite[Theorem 13.1]{paulsen_completely_2002} $S$ is completely isometrically
isomorphic to an operator system.

In view of the above observations we consider the Borel space $\widehat{\Xi }
$ of tuples 
\begin{equation*}
\xi =\left( f_{\xi },g_{\xi },h_{\xi },(C_{\xi,n })_{n\in \mathbb{N}},e_{\xi
}\right) \in \mathbb{N}^{\mathbb{N}\times \mathbb{N}}\times \mathbb{N}^{
\mathbb{Q}(i)\times \mathbb{N}}\times \mathbb{N}^{\mathbb{N}}\times
\prod_{n\in \mathbb{N}}2^{\mathbb{N}^{2}}\times \mathbb{N}
\end{equation*}
that code on $\mathbb{N}$ a $\mathbb{Q}(i)$-$\ast $-vector space structure $
S_{\xi }$ by setting
\begin{eqnarray*}
n+_{\xi }m &=&f_{\xi }\left( n,m\right) \text{,} \\
q\cdot _{\xi }n &=&g_{\xi }\left( q,n\right) \text{,} \\
n^{\ast _{\xi }} &=&h_{\xi }\left( n\right) \text{,}
\end{eqnarray*}
where $C_{\xi ,n}\subset M_{n}(\mathbb{N})$ is the positive cone, and $
e_{\xi }$ is the Archimedean matrix order unit. The corresponding norm on $
M_{n}(\mathbb{N})$ is denoted by $\left\Vert \cdot \right\Vert _{n,\xi }$.
The set $\widehat{\Xi }$ is Borel since the axioms defining a $\mathbb{Q}(i)$
-$\ast $-vector space are Borel conditions. The operator system $\mathcal{OS}
y(\xi )$ coded by $\xi $ is the completion of $S_{\xi }$. Note in particular
that the scalar multiplication can be uniquely extended on $\mathcal{OS}
y(\xi )$ from $\mathbb{Q}(i)$ to $\mathbb{C}$ and hence $\mathcal{OS}y(\xi )$
is indeed an operator system.

\subsubsection{Parametrizations as models of a theory\label{Subsubsection:
models}}

The Choi-Effros abstract characterization of operator systems allows one to
describe operator systems as models of a theory $\mathcal{T}_{OSy}$ in a
suitable language $\mathcal{L}_{OSy}$. The details can be found in \cite[
Appendix B]{goldbring_kirchberg_2014}. This allows one to define other
parametrizations of operator systems as models of $\mathcal{T}_{OSy}$ as in 
\cite[Subsection 3.4]{elliott_isomorphism_2013} or \cite[Section 1]
{ben_yaacov_lopez-escobar_2014}, and as Polish structures as in \cite[
Subsections 2.1 and 3.3]{elliott_isomorphism_2013}. Lemma 2.3 of \cite
{elliott_isomorphism_2013} together with the argument in \cite[Subsection 3.4
]{elliott_isomorphism_2013} show that these parametrizations are all weakly
equivalent to each other. Moreover they can be easily seen to be equivalent
to the parametrization $\widehat{\Xi }$ defined above using again \cite[
Lemma 2.3]{elliott_isomorphism_2013}. The analogous argument for C*-algebras
is presented in \cite[Section 3.1]{elliott_isomorphism_2013}.

\subsection{Parametrizations of operator spaces}

One can define parametrizations for the category $\mathbf{OSp}$ of operator
spaces in an analogous was for operator systems. To this purpose one can
consider $\mathcal{W}$ to be the $\mathbb{Q}(i)$-vector space of
noncommutative polynomials in the variables $X_{i}$ for $i\in \mathbb{N}$
and with no constant term. Fix an enumeration $(\mathfrak{q}_{n}\mathfrak{)}
_{n\in \mathbb{N}}$ of $\mathcal{W}$. The space $\Gamma $ is just the set of
sequences $B(H)^{\mathbb{N}}$. The space $\Gamma _{\mathcal{W}}$ is the set
of linear functions from $\mathcal{W}$ to $B(H)$. It can be easily verified
that $\Gamma $ and $\Gamma _{\mathcal{W}}$ are equivalent parametrizations.

The space $\Xi $ is the Borel set of $\delta =\left( \delta _{n}\right)
_{n\in \mathbb{N}}$, where $\delta _{n}\in \mathbb{R}^{M_{n}(W)}$ is such
that there is an operator system $X$ and a nonzero sequence $\gamma $ in $X$
such that
\begin{equation*}
\delta _{n}(\left[ p_{ij}\right] )=\left\Vert \left[ p_{ij}(\gamma )\right]
\right\Vert _{M_{n}(X)}
\end{equation*}
for every $n\in \mathbb{N}$ and $\left[ p_{ij}\right] \in M_{n}(W)$. One can
describe $\Xi $ as the set of $\delta $ such that setting 
\begin{equation*}
\left\Vert \left[ q_{k_{ij}}\right] \right\Vert =\delta _{n}(\left[ k_{ij}
\right] )
\end{equation*}
defines a nonzero \emph{operator seminorm structure }on $\mathcal{W}$; see 
\cite{blecher_operator_2004}. This makes it apparent that $\Xi $ is a Borel
set. The operator space $\mathcal{OS}p(\delta )$ associated with $\delta $
is the Hausdorff completion of $\mathcal{W}$---as described in \cite
{blecher_operator_2004}---with respect to the operator seminorm structure
induced by $\delta $.

Let us say that a $\mathbb{Q}(i)$-vector space $S$ is $L^{\infty }$
-matricially normed if, for every $n\in \mathbb{N}$, the space of matrices $
M_{n}(S)$ is endowed with a norm $\left\Vert \cdot \right\Vert _{n}$ such
that the following hold.

\begin{enumerate}
\item For every $n,m\in \mathbb{N}$, $x\in M_{n}(S)$, and $y\in M_{m}(S)$ 
\begin{equation*}
\left\Vert 
\begin{bmatrix}
x & 0 \\ 
0 & y
\end{bmatrix}
\right\Vert _{n,m}=\max \left\{ \left\Vert x\right\Vert _{n},\left\Vert
y\right\Vert _{m}\right\} \text{;}
\end{equation*}

\item For every $n,m,k\in \mathbb{N}$, $a\in M_{n,k}(\mathbb{Q}(i))$, $x\in
M_{k}(S)$, and $b\in M_{k,m}(\mathbb{Q}(i))$,
\begin{equation*}
\left\Vert axb\right\Vert _{m}\leq \left\Vert a\right\Vert \left\Vert
x\right\Vert _{k}\left\Vert b\right\Vert
\end{equation*}
where $\left\Vert a\right\Vert $ and $\left\Vert b\right\Vert $ are the
operator norms of $a$ and $b$ regarded as linear operators from $\mathbb{C}
^{k}$ to $\mathbb{C}^{n}$ and from $\mathbb{C}^{m}$ to $\mathbb{C}^{k}$.
\end{enumerate}

Define $\widehat{\Xi }$ to be the space of tuples 
\begin{equation*}
\xi =\left( f,g,\left( N_{n}\right) _{n\in \mathbb{N}}\right) \in \mathbb{N}
^{\mathbb{N}^{2}}\times \mathbb{N}^{\mathbb{Q}(i)\times \mathbb{N}}\times
\prod_{n\in \mathbb{N}}\mathbb{R}^{M_{n}(\mathbb{N})}
\end{equation*}
such that $\xi $ codes an $L^{\infty }$-matricially normed $\mathbb{Q}(i)$
-vector space structure $S_{\xi }$ on $\mathbb{N}$ by setting
\begin{align*}
n+_{\xi }m& =f_{\xi }\left( n,m\right) \text{,} \\
q\cdot _{g}m& =g_{\xi }\left( q,m\right) \text{, and} \\
\left\Vert \left[ m_{ij}\right] \right\Vert & =N_{n,m}\left( \left(
m_{ij}\right) \right)
\end{align*}
where $\left[ m_{ij}\right] $ is an $n\times m$ matrix. The fact that the
axioms of a $L^{\infty }$-matricially normed $\mathbb{Q}(i)$-vector spaces
are Borel conditions shows that $\widehat{\Xi }$ is a Borel set. The
operator space $\mathcal{OS}p(\xi )$ coded by $\xi $ is the completion of $
S_{\xi }$. The scalar multiplication on $\mathcal{OS}y(\xi )$ can be
uniquely extended from $\mathbb{Q}(i)$ to $\mathbb{C}$ and hence $\mathcal{OS
}p(\xi )$ is indeed an operator space.

Finally one can regard operator spaces as models of a theory in the logic
for metric structures; see \cite[Appendix B]{goldbring_kirchberg_2014}. This
gives other natural parametrizations for the category of operator spaces as
in \cite[Subsection 3.3]{elliott_isomorphism_2013}. The same observations as
the ones presented for operator systems in \ref{Subsubsection: models} apply.

\section{The classification of all separable operator systems and spaces 
\label{Section:separable}}

In this section we identify the complexity of the classification problem for
operator spaces. The corresponding problem for separable operator systems
has already been identified. By \cite[Theorem 1.1]{elliott_isomorphism_2013}
the complete order isomorphism relation of operator systems is Borel
reducible to a Polish group action. And it follows from \cite[Theorem 1.1]
{sabok_completeness_2013} that such a relation is in fact maximal among all
relations that lie below a Polish group action. (Observe that two
C*-algebras are *-isomorphic if and only if they are complete order
isomorphic by \cite[Corollary 1.3.10]{blecher_operator_2004}.) We now show
that the classification of operator spaces up to completely bounded
isomorphism is maximally complex among all analytic equivalence relations.
This has been independently observed by N.\ Christopher Phillips
(unpublished).

It is easy to see that the completely bounded isomorphism equivalence
relation is analytic (say in the parameterization $\Xi$). To show that it is
maximal among analytic equivalence relations, we need only find a Borel
reduction from another relation that is known to be complete. For this we
will use the isomorphism relation on Banach spaces. This latter relation is
defined on the standard Borel space $\mathfrak{B}$ of closed subspaces of $C
\left[ 0,1\right] $, endowed with the Effros Borel structure. (Recall that
any Banach space is isometrically isomorphic to a closed subspace of $C\left[
0,1\right] $.) It is shown in Theorem~5 of \cite{ferenczi_complexity_2009}
that the isomorphism relation on $\mathfrak{B}$ is indeed complete for
analytic equivalence relations.

\begin{theorem}
The classification problem for separable Banach spaces up to isomorphism is
Borel reducible to the classification problem for separable operator spaces
up to completely bounded isomorphism. As a consequence, the latter
equivalence relation is complete analytic.
\end{theorem}

\begin{proof}
Suppose that $X\in \mathfrak{B}$ is a closed subspace of $C\left[ 0,1\right] 
$. The minimal operator space structure on $X$ is defined by
\begin{equation*}
\left\Vert \left[ x_{ij}\right] \right\Vert =\sup \left\{ \left\Vert \left[
\varphi (x_{ij})\right] \right\Vert :\varphi \in X^{\prime }\text{, }
\left\Vert \varphi \right\Vert \leq 1\right\} \text{.}
\end{equation*}
Observe that by the Hahn--Banach theorem, it is equivalent to let $\varphi $
range over a weak*-dense subset of the unit ball of $C\left[ 0,1\right]
^{\prime }$. It is well known that two Banach spaces $X$ and $Y$ are
isomorphic if and only if they are completely isomorphic as operator spaces
when endowed with their minimal operator space structures; see \cite[1.2.21]
{blecher_operator_2004}. Therefore the assignment $X\mapsto X_{\min }$,
where $X_{\min }$ is the minimal operator structure on $X$, is a reduction
from the relation of isomorphism of Banach spaces to the relation of
complete isomorphism of operator spaces. We need only verify that this
reduction is given by a Borel function.

For this, fix a weak*-dense subset $D$ of the unit ball of $C\left[ 0,1
\right] ^{\prime }$. By the Kuratowski--Ryll-Nardzewski theorem \cite[
Theorem 12.13]{kechris_classical_1995} there is a sequence of Borel
functions $\sigma _{n}:\mathfrak{B}\rightarrow C\left[ 0,1\right] $ for $
n\in \mathbb{N}$ such that 
\begin{equation*}
\left\{ \sigma _{n}(X):n\in \mathbb{N}\right\}
\end{equation*}
is a dense subset of $X$ for every $X\in \mathfrak{B}$. For each $X\in 
\mathfrak{B}$ and $q\in \mathcal{W}$ define $q^{X}$ to be the element of $X$
obtained from $q$ replacing $X_{i}$ with $\sigma _{i}(X)$ for every $i\in 
\mathbb{N}$. We define the code $\delta _{X}\in X$ for the operator space $
X_{\min }$ by setting 
\begin{equation*}
\delta _{X}\left( \left[ q_{ij}\right] \right) =\sup \left\{ \left\Vert
\varphi (q_{ij}^{X})\right\Vert :\varphi \in D\right\} \text{.}
\end{equation*}
Then the function $X\mapsto \delta _{X}$ from $\mathfrak{B}$ to $\Xi $ is
Borel, as desired.
\end{proof}

\section{The classification of finitely generated operator systems\label
{Section:fg}}

We will show that, in contrast with the result from Section \ref{Section:separable}, the classification problem for finitely
generated operator systems is smooth. For convenience we will work in the
parametrization $\Gamma $. Fix $N\in \mathbb{N}$ and denote by $\Gamma
_{\leq N}$ the set of $\gamma \in \Gamma $ such that $\mathcal{OS}y(\gamma )$
has dimension at most $N$. It will be shown in Appendix \ref
{Appendix:equivalenceOSy} that $\Gamma _{\leq N}$ is a Borel subset of $
\Gamma $. Denote by $\Gamma _{N}$ the Borel set $\Gamma _{\leq N}\setminus
\Gamma _{\leq (N-1)}$, which provides a standard Borel parametrization of
operator systems of dimension $N$. Further define $\widehat{\Gamma }_{N}$ to
be the set of linearly independent tuples $\left( x_{1},\ldots ,x_{N}\right) 
$ such that $\mathrm{span}\left\{ x_{1},\ldots ,x_{N}\right\} $ is an
operator system. It will be shown in Appendix \ref{Appendix:equivalenceOSy}
that $\widehat{\Gamma }_{N}$ is a Borel subset of $B(H)^{N}$, and moreover
the parametrizations $\Gamma _{N}$ and $\widehat{\Gamma }_{N}$ of $N$
-dimensional operator systems are weakly equivalent.

\subsection{The classification of all finitely generated operator systems}

In this subsection we consider the classification problem for arbitrary
finitely generated operator systems.

\begin{theorem}
\label{Theorem: iso fg osy}The relation of complete order isomorphism of
finitely generated operator systems is smooth.
\end{theorem}

Theorem \ref{Theorem: iso fg osy} can be seen as a generalization of the
main results of \cite{arveson_noncommutative_2010}, where Arveson provided a
concrete classification for operator systems with finite-dimensional
C*-envelope.

In order to prove Theorem \ref{Theorem: iso fg osy} it is enough to show
that for every $N\in \mathbb{N}$ such relation is smooth when restricted to
operator systems of dimension $N$. For convenience we will work in the
parametrization for $N$-dimensional operator systems $\widehat{\Gamma }_{N}$
. We proceed to define a compact metrizable space of complete isometry
classes of $N$-dimensional operator systems. 

The argument is analogous to the one for operator spaces from \cite[Chapter
21]{pisier_introduction_2003}---see in particular Remark 21.2, Lemma 21.7,
and Exercise 21.1 from \cite{pisier_introduction_2003}. Suppose that $X$ and 
$Y$ are $N$-dimensional operator systems with Archimedean matrix order units 
$e_{X}$ and $e_{Y}$. Fix $n\in \mathbb{N}$ and define $d_{n}\left(
X,Y\right) $ to be the infimum of
\begin{equation*}
\max \left\{ ||u\left( e_{X}\right) -e_{Y}||,\log ||id_{M_{n}(\mathbb{C}
)}\otimes u||,\log ||id_{M_{n}(\mathbb{C})}\otimes u^{-1}||\right\}
\end{equation*}
where $u$ ranges over all isomorphisms from $X$ to $Y$.

\begin{lemma}
\label{Lemma: separates}If $X$ and $Y$ are $N$-dimensional operator systems,
then $X$ and $Y$ are completely order isomorphic if and only if $d_{n}\left(
X,Y\right) =0$ for every $n\in \mathbb{N}$.
\end{lemma}

\begin{proof}
Necessity is obvious. Conversely suppose that $d_{n}\left( X,Y\right) =0$
for every $n\in \mathbb{N}$. Therefore for every $n\in \mathbb{N}$ there is
an isomorphism $u_{n}:X\rightarrow Y$ such that 
\begin{equation*}
||id_{M_{n}(\mathbb{C})}\otimes u_n||<1+2^{-n}\text{,}
\end{equation*}
\begin{equation*}
||id_{M_{n}(\mathbb{C})}\otimes u_n^{-1}||<1+2^{-n}\text{,}
\end{equation*}
and
\begin{equation*}
||u_n\left( e_{X}\right) -e_{Y}||<2^{-n}\text{.}
\end{equation*}
Fix a basis $\left( x_{1},\ldots ,x_{N}\right) $ for $X$. After passing to a
subsequence we can assume that for every $i\leq N$ the sequence $\left(
u_{n}\left( x_{i}\right) \right) _{n\in \mathbb{N}}$ converges in $Y$ to some $
y_{i}$. Define $u$ to be the linear function from $X$ to $Y$ sending $x_{i}$
to $y_{i}$ for every $i\leq N$. From the fact that a unital map is completely positive if and only if it is completely contractive, we deduce that $u$ is a complete order
isomorphism from $X$ to $Y$.
\end{proof}

By Lemma \ref{Lemma: separates} we can consider the space $OSy(N)$ of
complete order isomorphism classes of $N$-dimensional operator systems
endowed with the topology induced by the metric
\begin{equation*}
\delta _{w}\left( X,Y\right) =\sum_{n\in \mathbb{N}}2^{-n}d_{n}\left(
X,Y\right) \text{.}
\end{equation*}

\begin{lemma}
The space $OSy(N)$ is compact.
\end{lemma}

\begin{proof}
Suppose that $\left( X_{k}\right) _{k\in \mathbb{N}}$ is a sequence of $N$
-dimensional operator systems. By \cite[Theorem 4.13]{bollobas_linear_1999}
for every $k\in \mathbb{N}$ one can find a normalized linear basis $
\overline{b}^{\left( k\right) }$ for $X_{k}$ such that the dual basis is
also normalized. Observe that this implies that
\begin{equation*}
\left\Vert \left( \lambda _{1},\ldots ,\lambda _{N}\right) \right\Vert
_{\ell ^{1}_N}=\sum_{i\leq N}\left\vert \lambda _{i}\right\vert \leq
N\,\|(\lambda_1,\ldots,\lambda_N)\|_{\ell^\infty_N}\leq
N\left\Vert \sum_{i\leq N}\lambda _{i}b_{i}^{\left( k\right) }\right\Vert
\end{equation*}
for $\lambda _{1},\ldots ,\lambda _{n}\in \mathbb{C}$. The basis $\overline{b
}^{\left( k\right) }$ induces a linear isomorphism of $X_{k}$ with $\mathbb{C
}^{N}$. Thus we can assume without loss of generality that the support of $
X_{k}$ is $\mathbb{C}^{N}$. Observe that $\left\Vert x\right\Vert _{k}\leq
\left\Vert x\right\Vert _{\ell ^{1}_N}$ for every $x\in X_{k}$. Denote by $
\Omega $ the unit ball of $\ell ^{1}_N$ and let  $\overline{e}^{\left(
k\right) }$  be the order unit of $X_{k}$ for every $k\in \mathbb{N}$. The
functions $\left\Vert \cdot \right\Vert _{k}$ are equiuniformly continuous
on $\Omega $. Therefore by the Arzel\'{a}-Ascoli theorem one can assume, after passing to a
subsequence, that the sequence 
\begin{equation*}
\left( \left\Vert \cdot \right\Vert _{k}\right) _{k\in \mathbb{N}}
\end{equation*}
converges uniformly on $\Omega $, and moreover $e_{i}^{\left( k\right) }$
converges for every $i\leq n$. Denote by $\left\Vert \cdot \right\Vert
_{\infty }$ and $e_{i}^{\left( \infty \right) }$ the corresponding limits.
If $Y$ is the corresponding Banach space, then the construction ensures that 
$X_{k}\rightarrow Y$ in the Banach Mazur distance for Banach spaces. A
similar argument can be applied to the Banach spaces $M_{n}\left(
X_{k}\right) $ for every $n\in \mathbb{N}$, every time passing to a further
subsequence. Finally take a diagonal subsequence and denote by $X_{\infty }$
the $L_{\infty }$-matrix-norm structure on $\mathbb{C}^{N}$ with
distinguished element $\overline{e}^{\left( \infty \right) }$ obtained as a
limit. For every $k\in \mathbb{N}$ fix a complete order embedding $u_{k}$ of 
$X_{k}$ into a unital C*-algebra $A_{k}$. Fix a nonprincipal ultrafilter $
\mathcal{U}$ over $\mathbb{N}$ and define $A$ to be the ultraproduct $\prod_{
\mathcal{U}}A_{k}$. Define the function $u:\mathbb{C}^{N}\rightarrow A$ by
\begin{equation*}
u(x)=\lim_{k\rightarrow \mathcal{U}}u_{k}(x)\text{.}
\end{equation*}
It is not difficult to verify that $u$ is a unital completely isometric
embedding of $X_{\infty }$ into $A$. This shows that $X_{\infty }$ is an
operator system, which is by construction the limit of the sequence $\left(
X_{k}\right) _{k\in \mathbb{N}}$.
\end{proof}

It only remains to show that the function from $\widehat{\Gamma }_{N}$ to $
OSy(N)$ assigning to $\bar{x}$ the complete order isomorphism class of $
\mathrm{span}(\bar{x})$ is Borel. Denote by $\mathcal{W}_{N}$ the set of
polynomials of degree $1$ in the noncommutative variables $X_{1},\ldots
,X_{N}$ and with coefficients from $\mathbb{Q}(i)$. Similarly denote by $
\mathcal{W}_{\mathbb{C},N}$ the set of polynomials of degree $1$ in the
noncommutative variables $X_{1},\ldots ,X_{N}$ and with coefficients from $
\mathbb{C}$.

\begin{lemma}
\label{Lemma: select e}There is a Borel function $\bar{x}\mapsto p_{e}^{\bar{
x}}$ from $\widehat{\Gamma }_{N}$ to $\mathcal{W}_{\mathbb{C},N}$ such that $
p_{e}^{\bar{x}}(\bar{x})=I$.
\end{lemma}

\begin{proof}
Fix $k\in \mathbb{N}$. Denote by $\widehat{\Gamma }_{N,k}$ the Borel set of $
\bar{x}\in \widehat{\Gamma }_{N}$ such that for every $\varepsilon >0$ there
is $p\in \mathcal{W}_{N}$ such that $\left\Vert p\right\Vert \leq k$ and $
\left\Vert p\left( x_{1},\ldots ,x_{N}\right) -I\right\Vert <\varepsilon $.
Observe that the relation 
\begin{equation*}
\left\{ \left( \bar{x},p\right) \in \widehat{\Gamma }_{N,k}\times \mathcal{W}
_{\mathbb{C},N}:p(\bar{x})=I\text{, }\left\Vert p\right\Vert \leq k\right\}
\end{equation*}
is Borel and has compact sections. Then the conclusion follows from the
Kuratowski--Ryll-Nardzewski theorem \cite[Theorem 12.13]
{kechris_classical_1995}.
\end{proof}

We can finally present the proof of Theorem \ref{Theorem: iso fg osy}.

\begin{proof}[Proof of Theorem \protect\ref{Theorem: iso fg osy}]
As observed earlier, it is enough to show that for every $N\in \mathbb{N}$
the relation of complete order isomorphism of $N$-dimensional operator
systems is smooth. For convenience we will work in the parametrization $%
\widehat{\Gamma }_{N}$. Consider the compact metrizable space $OSy(N)$
defined above. The map from $\widehat{\Gamma }_{N}$ to $OSy(N)$ that assigns
to an element $\left( x_{1},\ldots ,x_{N}\right) $ of $\widehat{\Gamma }_{N}$
the class of $\mathrm{span}\left( \overline{x}\right) $ is a reduction from
the relation of complete order isomorphism to the relation of equality in $%
OSy(N)$. In order to conclude that the former relation is smooth it is
enough to show that such a reduction is Borel. To this purpose, fix $\bar{x}%
\in \widehat{\Gamma }_{N}$, $n\in \mathbb{N}$, and $\varepsilon >0$. It is
enough to show that the set of $\bar{y}\in \widehat{\Gamma }_{N}$ such that%
\begin{equation*}
d_{n}\left( \mathrm{span}(\bar{x}),\mathrm{span}(\bar{y})\right)
<\varepsilon 
\end{equation*}%
is Borel. Observe $\bar{y}$ belongs to such a set if and only if there are $%
p_{i}\in \mathcal{W}_{N}$ such that

\begin{enumerate}
\item $\left\Vert p_{e}^{\bar{x}}(\bar{y})-I\right\Vert <\varepsilon $;

\item $\left\Vert \left[ p_{ij}(\bar{y})\right] \right\Vert \leq \left(
1+\varepsilon \right) \left\Vert \left[ p_{ij}(x)\right] \right\Vert $ for
every $p_{ij}\in \mathcal{W}_{N}$;

\item $\left\Vert \left[ p_{ij}(x)\right] \right\Vert \leq \left(
1+\varepsilon \right) \left\Vert \left[ p_{ij}(\bar{y})\right] \right\Vert $
for every $p_{ij}\in \mathcal{W}_{N}$.
\end{enumerate}

By virtue of Lemma \ref{Lemma: select e} these are Borel conditions. 
\end{proof}

\subsection{Operator systems generated by a single unitary\label{Subsection:
single unitary}}

In this subsection we provide a concrete classification of operator systems
generated by a single unitary operator $U$ in terms of the spectrum $\sigma
(U)$ of $U$. First of all we observe that the set $\Gamma _{U}$ of $\gamma
\in \Gamma $ such that $\mathcal{OS}y(\gamma )$ is generated by a single
unitary is Borel. In fact $\gamma \in \Gamma _{U}$ if and only if $\gamma
\in \Gamma _{\leq 3}$ and moreover there is $n\in \mathbb{N}$ such that for
every $m\in \mathbb{N}$ and $q_{1},\ldots ,q_{m}\in \mathcal{V}$ there is $
p\in \mathcal{V}_{n}$ such that $\left\Vert p\right\Vert \leq n$, $
\left\Vert \left( p^{\ast }p\right) (\gamma )-1\right\Vert <\frac{1}{m}$ and
for every $j\leq m$ there is $r\in \mathcal{V}$ such that
\begin{equation*}
\left\Vert r\left( p(\gamma )\right) -q_{j}(\gamma )\right\Vert <1/m\text{.}
\end{equation*}
This argument together with \cite[Theorem 28.8]{kechris_classical_1995}
shows that there is a Borel map $\gamma \mapsto p^{\gamma }$ from $\Gamma
_{U}$ to $\mathcal{V}_{\mathbb{C}}$ such that $p^{\gamma }(\gamma )$ is a
unitary generator of $\Gamma _{U}$. This also shows that the standard Borel
space $\mathcal{U}(H)$ of unitary operators on $H$ is a weakly equivalent
parametrization for the category of operator systems generated by a single
unitary. (Observe that $\mathcal{U}(H)$ is a $G_{\delta }$ subset of the unit ball of $
B(H)$ with respect to the weak operator topology, and hence a standard Borel
space.)

In the following we will consider the parametrization $\mathcal{U}(H)$. The operator
system $\mathcal{OS}y(U)$ coded by $U$ is the closed linear span of the set $
\left\{ 1,U,U^{\ast }\right\} $. Recall that the C*-envelope of an operator
system is in some sense the minimal C*-algebra containing a given operator
system. It was introduced in terms of boundary representations in \cite
{arveson_subalgebras_1969}. The \emph{\v Silov ideal} of an operator system is the intersection of the kernels of all of its boundary representations; when this ideal is trivial, we say that the system is  \emph{reduced}, and the C$^*$-algebra it generates is its C$^*$-envelope. 

An abstract proof of the existence of the C$^*$-envelope, making no
reference to boundary representations, was given by Hamana using his theory
of injective envelopes \cite{hamana_injective_1979,hamana_injective_1979-1}.
The existence of sufficiently many boundary representations to determine the
C*-envelope was finally established in the separable case in \cite
{arveson_noncommutative_2008} and in full generality in \cite
{davidson_choquet_2013}. Combining this with Arveson's initial work we have the following:

\begin{theorem}\label{theorem: coi extends to unitary iso}
Let $\mathcal S_1\subset C^*(\mathcal S_1)$ and $\mathcal S_2\subset C^*(\mathcal S_2)$ be two reduced operator systems, and let $\phi:\mathcal S_1\to\mathcal S_2$ be a complete order isomorphism. Then there exists a $*$-isomorphism $\tilde\phi:C^*(\mathcal S_1)\to C^*(\mathcal S_2)$, with $\tilde\phi|_{\mathcal S_1}=\phi$. 
\end{theorem}
\begin{proof}
This follows from \cite[Theorem 2.2.5]{arveson_subalgebras_1969} and \cite[Theorem 3.4]{davidson_choquet_2013}.
\end{proof}

For the reader's convenience, we include the following well-known fact. 
\begin{lemma}
\label{Lemma: envelope}Suppose that $V_{1},\ldots ,V_{n}\in B(H)$. Define $X$
to be the operator system generated by
\begin{equation*}
\left\{ V_{1},\ldots ,V_{n},V_{1}^{\ast }V_{1},\ldots ,V_{n}^{\ast
}V_{n},V_{1}V_{1}^{\ast },\ldots ,V_{n}V_{n}^{\ast }\right\} \text{.}
\end{equation*}
Then the C*-envelope of $X$ can be identified with the C*-algebra $C^{\ast
}(X)$ generated by $X$ inside $B(H)$. 
\end{lemma}

\begin{proof}
We will show that every (irreducible) representation of $C^*(X)$ is a boundary representation, so the \v Silov ideal of $X$ is trivial, implying the result. We actually show that $X$ is hyperrigid, but we don't really need this fact. Let $\pi : C^*(X) \to B(H)$ be
a representation. We must show that if $\phi : C^*(X) \to B(H)$ is a
completely positive extension of the restriction $\pi|_X$, then $\phi = \pi$.
We have for every $i = 1,\ldots,n$, 
\begin{equation*}
\phi(V_i^*) \phi(V_i) = \pi(V_i^*)\pi(V_i^*) = \pi(V_i^*V_i)= \phi(V_i^* V_i),
\end{equation*}
and similarly $\phi(V_i)\phi(V_i^*) = \phi(V_i V_i^*)$. Thus $V_1,\ldots,V_n$
each belong to the multiplicative domain of $\phi$ (see e.g. \cite[
Proposition 1.3.11]{blecher_operator_2004}), and so $\phi$ is multiplicative. Since $\phi(V_i) = \pi(V_i)$
for each $i = 1,\ldots,n$, it follows that $\phi = \pi$.
\end{proof}

We can now address the announced classification of operator
systems generated by a single unitary.

\begin{theorem}
\label{Theorem: classification single unitary}Suppose that $U,V\in \mathcal{U}(H)$
and $\sigma (V)$ has at least $5$ points. Then the  following statements are
equivalent:

\begin{enumerate}
\item there exists a $\ast $-isomorphism $\pi :C^{\ast }(U)\rightarrow
C^{\ast }(V)$ with $\pi (U)=\lambda V$ or $\pi (U)=\lambda V^{\ast }$ for
some $\lambda \in \mathbb{T}$;

\item $\sigma (U)=\lambda \sigma (V)$ or $\sigma (U)=\lambda \overline{
\sigma (V)}$ for some $\lambda \in \mathbb{T}$;

\item $\mathcal{OS}y(U)$ is completely order isomorphic to $\mathcal{OS}y(V)$
.
\end{enumerate}
If $\sigma(V)$ has 3 points or less, then $\mathcal{OS}y(U)$ is completely order isomorphic to $\mathcal{OS}y(V)$ if and only if $|\sigma(U)|=|\sigma(V)|$.
\end{theorem}

\begin{proof}
The implication (1)$\implies $(2) is clear by the spectral mapping theorem. 
For (2)$\implies $(3),
let $W=\lambda V$. Then the hypothesis is that $\sigma (U)=\sigma (W)$. Thus 
$C^{\ast }(U)\simeq C(\sigma (U))=C(\sigma (W))\simeq C^{\ast }(W)$ as
C*-algebras, via a *-isomorphism $\pi $ such that $\pi (U)=W$ (because the
two isomorphisms above map $U\mapsto z\mapsto W$). If we let $\varphi =\pi
|_{\mathcal{OS}y(U)}$, we get a unital completely positive map $\varphi :
\mathcal{OS}y(U)\rightarrow \mathcal{OS}y(W)$ with 
\begin{equation*}
\varphi (\alpha I+\beta U+\gamma U^{\ast })=\alpha I+\beta W+\gamma W^{\ast
}.
\end{equation*}
This is well-defined and bijective, because of the fact that $|\sigma
(U)|=|\sigma (W)|=|\sigma (V)|>2$ implies that $I,U,U^{\ast }$ are linearly
independent. We have $\varphi ^{-1}=\pi ^{-1}|_{\mathcal{OS}y(V)}$, so it is
also unital completely positive. Thus, $\mathcal{OS}y(U)$ is completely
order isomorphic to $\mathcal{OS}y(W)=\mathcal{OS}y(V)$. Finally, for the
case where $\sigma (U)=\lambda \overline{\sigma (V)}$, we take $W=\lambda
V^{\ast }$ and repeat the argument above.

Now we show that (3)$\implies$(1). Suppose that $\mathcal{OS}y(U)$ is
completely order isomorphic to $\mathcal{OS}y(V)$. By Lemma \ref{Lemma:
envelope} the C*-envelope $C_{e}^{\ast }\left( \mathcal{OS}y(U)\right) $
coincides with the C*-algebra $C^{\ast }(U)$ generated by $U$ inside $B(H)$.
The same applies to $V$. By Theorem \ref{theorem: coi extends to unitary iso}, there is a *-isomorphism $
\varphi $ from $C^{\ast }(U)$ to $C^{\ast }(V)$ that sends $\mathcal{OS}y(U)$
onto $\mathcal{OS}y(V)$. In particular
\begin{equation*}
\varphi (U)=\alpha I+\beta V+\gamma V^{\ast }
\end{equation*}
for some $\alpha ,\beta ,\gamma \in \mathbb{C}$. Since $U$ is unitary and $
\varphi $ is a $*$-homomorphism we have that
\begin{align*}
I=\varphi \left( U^{\ast }U\right) & =\left( \left\vert \alpha \right\vert
^{2}+\left\vert \beta \right\vert ^{2}+\left\vert \gamma \right\vert
^{2}\right) I+\left( \overline{\alpha }\beta +\alpha \overline{\gamma }
\right) V+\left( \alpha \overline{\beta }+\overline{\alpha }\gamma \right)
V^{\ast } \\
& \ \ +\overline{\beta }\gamma V^{2}+\beta \overline{\gamma }(V^{\ast })^{2}.
\end{align*}
The assumption that $\sigma (V)$ has at least $5$ points
implies that the elements
$
\left\{ I,V,V^{\ast },V^{2},\left( V^{\ast }\right) ^{2}\right\}
$
of $B(H)$ are linearly independent (because $C^*(V)$ is at least 5-dimensional). Therefore we get the relations
\begin{eqnarray*}
1 &=&\left\vert \alpha \right\vert ^{2}+\left\vert \beta \right\vert
^{2}+\left\vert \gamma \right\vert ^{2}\text{,} \\
0 &=&\overline{\alpha }\beta +\alpha \overline{\gamma }\text{,} \\
0 &=&\alpha \overline{\beta }+\overline{\alpha }\gamma \text{,} \\
0 &=&\overline{\beta }\gamma \text{.}
\end{eqnarray*}
Exactly one between $\beta $ and $\gamma $ is zero (if both
were zero,   we would get $\varphi (U)=\alpha I$, and the image of $
\varphi $ would be one-dimensional). This forces $\alpha=0 $ and $|\beta|=1$, $\gamma=0$; or $\beta=0$ and $|\gamma|=1$.

Finally, if $\sigma (U)$ has at most $3$ points then $\mathcal{OS}y(U)=C_{e}^{\ast }\left( \mathcal{OS}y(U)\right)$. So the complete order isomorphism of the operator systems agrees with the isomorphism of the C$^*$-algebras, which is given by homeomorphism of the spectra. 
\end{proof}

\begin{remark}
Denote by $\mathcal{K}\left( \mathbb{T}\right) $ the space of closed subsets
of $\mathbb{T}$ endowed with the Vietoris topology \cite[4.F]
{kechris_classical_1995}. By \cite[Theorem 1.1]{latrach_facts_2005} the
function $U\mapsto \sigma (U)$ assigning to $U\in \mathcal{U}(H)$ its spectrum is
Borel. Denote by $\mathrm{\mathrm{Iso}m}\left( \mathbb{T}\right) $ the group
of isometries of $\mathbb{T}$ endowed with the topology of pointwise
convergence. Elements of $\mathrm{\mathrm{\mathrm{Iso}m}}\left( \mathbb{T}
\right) $ are of the form $z\mapsto \lambda z$ or $z\mapsto \lambda \bar{z}$
for $\lambda \in \mathbb{T}$. The group of isometries of $\mathbb{T}$ has a
natural continuous action on $\mathcal{K}\left( \mathbb{T}\right) $. Denote
by $E_{\mathrm{Iso}\left( \mathbb{T}\right) }^{\mathcal{K}\left( \mathbb{T}
\right) }$ the corresponding orbit equivalence relation. Observe that, being
the orbit equivalence relation of a continuous action of a compact group, $
E_{\mathrm{Iso}\left( \mathbb{T}\right) }^{\mathcal{K}\left( \mathbb{T}
\right) }$ is in particular smooth. Theorem \ref{Theorem: classification single unitary} shows that $U\mapsto \sigma (U)$ is a Borel
reduction from the relation of complete order isomorphism of operator
systems generated by a single unitary with at least $5$ points in the
spectrum to $E_{\mathrm{Iso}\left( \mathbb{T}\right) }^{\mathcal{K}\left( 
\mathbb{T}\right) }$.
\end{remark}

We do not know exactly what happens when $
\sigma (U)$ has exactly $4$ points. The following example at least shows
that the operator systems $\mathcal{OS}y(U)$ for $U$ unitary with four-point
spectrum are not all isomorphic, so the behaviour seems to be more similar to the $5+$ case. Consider the unitary elements 
\begin{equation*}
U=
\begin{bmatrix}
1 & 0 & 0 & 0 \\ 
0 & -1 & 0 & 0 \\ 
0 & 0 & i & 0 \\ 
0 & 0 & 0 & -i
\end{bmatrix}
\text{,}\quad V=
\begin{bmatrix}
1 & 0 & 0 & 0 \\ 
0 & \frac{1+i}{\sqrt{2}} & 0 & 0 \\ 
0 & 0 & i & 0 \\ 
0 & 0 & 0 & -1
\end{bmatrix}
\end{equation*}
of $M_{4}(\mathbb{C})$. Then 
$
C_{e}^{\ast }(\mathcal{OS}y(U))=C^{\ast }(U)=C^{\ast }(V)=C_{e}^{\ast }(
\mathcal{OS}y(V))
$
is the diagonal MASA of $M_{4}(\mathbb{C})$. Suppose by contradiction that $
\mathcal{OS}y(U)\cong \mathcal{OS}y(V)$. Then Theorem \ref{theorem: coi extends to unitary iso} there is a *-isomorphism $\pi $ from $C^{\ast
}(U)$ onto $C^{\ast }(V)$ mapping $\mathcal{OS}y(U)$ onto $\mathcal{OS}y(V)$
. Then $\pi (U)\in \mathcal{OS}y(V)$ is a diagonal matrix with eigenvalues $
\{1,-1,i,-i\}$ that is in the span of $\{I,V,V^{\ast}\}$. 
Consider the determinant 
\begin{equation*}
\begin{vmatrix}
a & 1 & 1 & 1 \\ 
b & 1 & \frac{1+i}{\sqrt{2}} & \frac{1-i}{\sqrt{2}} \\ 
c & 1 & i & -i \\ 
d & 1 & -1 & -1
\end{vmatrix}
=2i(-a+2b-\sqrt{2}c+(\sqrt{2}-1)d).
\end{equation*}
Observe that no choice of $a,b,c,d$ with $\{a,b,c,d\} = \{1,-1,i,-i\}$ can
make the above determinant equal to zero. This shows that $\pi (U)$ is
linearly independent from $I,V,V^{\ast }$, a contradiction.

\subsection{Uncountably many classes}

We now show that the relation of complete order isomorphism of finitely
generated operator systems has uncountably many classes. In fact there are
uncountably many classes even when restricting to operator systems whose
C*-envelope is $M_{3}(\mathbb{C})$.  For $t\in (0,1]$ consider the operator
\begin{equation*}
W_{t}=
\begin{bmatrix}
0 & 0 & 0 \\ 
1 & 0 & 0 \\ 
0 & t & 0
\end{bmatrix}
\text{.}
\end{equation*}
Observe that $W_{t}$ is irreducible and hence the C*-algebra $C^{\ast
}\left( W_{t}\right) $ generated by $W_{t}$ coincides with $M_{3}(\mathbb{C}
) $. The C*-envelope can always be realized as a quotient of $C^{\ast }\left(
W_{t}\right) $. By simplicity of $M_{3}(\mathbb{C})$ we must have $
C_{e}^{\ast }\left( W_{t}\right) =C^{\ast }\left( W_{t}\right) =M_{3}(
\mathbb{C})$. Denote by $X_{t}$ the operator system $\mathcal{OS}y\left(
W_{t}\right) $ generated by $W_{t}$. We claim that $X_{t}$ and $X_{s}$ are
complete order isomorphic if and only if $s=t$. Suppose that $X_{t}$ and $
X_{s}$ are complete order isomorphic. As mentioned above, the simplicity of $M_3(\mathbb{C})$ makes both $X_t$ and $X_s$ to be reduced. By Theorem \ref{theorem: coi extends to unitary iso}, there is a $*$-automorphism of $M_3(\mathbb{C})$ that maps $X_t$ onto $X_s$. It is well-known that $*$-automorphisms of finite-dimensional C$^*$-algebras are implemented by unitaries, so 
there is a unitary $
U\in M_{3}(\mathbb{C})$ such that the function $\Phi :A\mapsto UAU^{\ast }$
sends $X_{t}$ onto $X_{s}$. In particular
\begin{equation*}
\Phi \left( W_{t}\right) =\alpha I+\beta W_{s}+\gamma W_{s}^{\ast }
\end{equation*}
for some $\alpha,\beta,\gamma\in\mathbb{C}$. 
Denote by $\tau $ the  trace on $M_{3}(\mathbb{C})$. Since $
\Phi $ is trace-preserving we have that
\begin{equation*}
0=\tau \left( W_{t}\right) =\tau \left( \Phi \left( W_{t}\right) \right)
=\alpha \text{.}
\end{equation*}
Also
\begin{equation*}
0=\tau \left( W_{t}^{2}\right) =\tau (\Phi \left( W_{t}\right) ^{2})=
2(1+s^{2})\beta \gamma \text{.}
\end{equation*}
Therefore exactly one between $\beta $ and $\gamma $ is zero (since $W_t\ne0$). Suppose that $
\gamma =0$. In this case, the singular values of $W_t$ are $\{0,1,t\}$ and the singular values of $\Phi(W_t)=\beta W_s$ are $\{0,|\beta|,|\beta|\,s\}$. So either $|\beta|=1$ and $s=t$, or $|\beta|=t$ and $|\beta|\,s=1$. As both $|\beta|\leq1$ and $s\leq1$, this requires $|\beta|=s=1$, and then $t=|\beta|=1=s$.  In a similar way we can conclude $s=t$ in the case $\beta=0$. 

Therefore $\left( X_{t}\right) _{t\in (0,1]}$ is a one-parameter family of
pairwise not completely order isomorphic operator systems with the same
C*-envelope $M_{3}(\mathbb{C})$.

Using arguments similar to those above---just a little more involved---we can also produce an uncountable family of non-isomorphic operator systems with $C^*$-envelope $M_2(\mathbb{C})$, by taking $W_t=\begin{bmatrix}1&0\\ t&0\end{bmatrix}$. 

\subsection{Compact subsets of $\mathbb{C}^{n}$}

Denote by $\mathcal{K}(\mathbb{C})$ the space of compact subsets of $\mathbb{
C}$ endowed with the Effros Borel structure. It can be essentially deduced
from results of \cite{hjorth_classification_2000} that the relation of
homeomorphism of compact subsets of $\mathbb{C}$ is not classifiable by
countable structures. (This is an unpublished observation of Farah-Toms-T
\"{o}rnquist.) Moreover it is an open problem whether this relation has in
fact maximal complexity among all the relations that are classifiable by the
orbits of a continuous action of a Polish group on a Polish space. In this
direction, it was recently shown by Zielinski that the relation of
homeomorphism of compact metrizable spaces has indeed maximal complexity 
\cite{zielinski_complexity_2014}.

In this subsection we consider another related equivalence relation on $
\mathcal{K}(\mathbb{C})$.

\begin{definition}
Suppose that $D,\widetilde{D}$ are closed subsets of $\mathbb{C}$. A \emph{
degree }$1$ \emph{map }(in $z$ and $\bar{z}$ separately) on $D$ is a
function $f$ from $D$ to $\mathbb{C}$ of the form
\begin{equation}
f(z)=\alpha +\beta z+\gamma \bar{z}+\delta z\bar{z}\text{\label
{Equation:degree1-1}}
\end{equation}
for some $\alpha ,\beta ,\gamma ,\delta \in \mathbb{C}$ such that moreover
for some $\alpha ^{\prime },\beta ^{\prime },\gamma ^{\prime },\delta
^{\prime }\in \mathbb{C}$,
\begin{equation}
f(z)\overline{f(z)}=\alpha ^{\prime }+\beta ^{\prime }z+\gamma ^{\prime
}z+\delta ^{\prime }z\bar{z}\text{\label{Equation:degree1-2}}
\end{equation}
for every $z\in D$. A\emph{\ degree }$1$\emph{\ homeomorphism} from $D$ to $
\widetilde{D}$ is a homeomorphism $\varphi :D\rightarrow \widetilde{D}$ such
that $\varphi $ and $\varphi ^{-1}$ are degree $1$ maps. The closed subsets $
D$ and $\widetilde{D}$ of $\mathbb{C}$ are \emph{degree }$1$ \emph{
homeomorphic }if there is a degree $1$ homeomorphism from $D$ to $\widetilde{
D}$.
\end{definition}

The natural definition of degree $1$   map in $z$ and $\bar{z}$
would only involve Condition \ref{Equation:degree1-1}. However, Condition 
\ref{Equation:degree1-2} is necessary to ensure that the relation of degree $
1$ homeomorphism be an equivalence relation. In the rest of this subsection
we will prove Theorem \ref{Theorem: degree 1 dimension 1} below.  

We need to  relate degree $1$ homeomorphism
with complete order isomorphism of operator systems. Suppose that $V,
\widetilde{V}$ are bounded normal linear operators on $X$. Define $X=
\mathcal{OS}y\left( V,VV^{\ast }\right) $, $\widetilde{X}=\mathcal{OS}y(
\widetilde{V},\widetilde{V}\widetilde{V}^{\ast })$. Denote by $A$ and $
\widetilde{A}$ the (commutative) C*-algebras $C^{\ast }(V)$ and $C^{\ast }(
\widetilde{V})$, and denote by $D$ and $\widetilde{D}$ their spectra (which
coincide with the spectra of $V$ and $\widetilde{V}$). The proof of the
following lemma is immediate.

\begin{lemma}
\label{Lemma: degree 1}Under the identifications $A\cong C(D)$ and $
\widetilde{A}\cong C(\widetilde{D})$, isomorphisms from $A$ to $\widetilde{A}
$ mapping $X$ onto $\widetilde{X}$ correspond to degree $1$ homeomorphisms
from $D$ to $\widetilde{D}$.
\end{lemma}

Using Lemma \ref{Lemma: degree 1} we can relate degree $1$ homeomorphism of
compact sets with complete order isomorphism of operator systems.

\begin{lemma}
\label{Lemma: iso osy normal}With the notations above, the following
statements are equivalent:

\begin{enumerate}
\item $X$ and $\widetilde{X}$ are completely order isomorphic;

\item there is an isomorphism from $A$ to $\widetilde{A}$ mapping $X$ onto $
\widetilde{X}$;

\item $D$ and $\widetilde{D}$ are degree $1$ homeomorphic.
\end{enumerate}
\end{lemma}

\begin{proof}
By Lemma \ref{Lemma: envelope} the C*-envelopes of $X$ and $\widetilde{X}$
can be identified with $A$ and $\widetilde{A}$. Therefore the equivalence of
(1) and (2) follows from Theorem \ref{theorem: coi extends to unitary iso}.
The equivalence of (2) and (3) is a consequence of Lemma \ref{Lemma: degree
1}.
\end{proof}

Now we can state and prove the aforementioned result. 

\begin{theorem}
\label{Theorem: degree 1 dimension 1}The relation of degree $1$
homeomorphism of subsets of $\mathbb{C}$ is smooth.
\end{theorem}
\begin{proof}
Working in the parametrization $\Xi $, it is clear that there is a Borel
function from $\mathcal{K}(\mathbb{C})$ to $\Xi $ that assigns to a compact
subset $X$ of $\mathbb{C}$ a code $\xi _{X}$ for the operator system
generated by the identity function inside the C*-algebra $C(X)$ of
continuous complex-valued functions on $X$. By Lemma \ref{Lemma: iso osy
normal} this is a Borel reduction from the relation of degree $1$
homeomorphism of compact subsets of $\mathbb{C}$ to the relation of complete
order isomorphism of finitely generated operator systems. In view of Theorem 
\ref{Theorem: iso fg osy} in particular this shows that the relation of
degree $1$ homeomorphism of compact subsets of $\mathbb{C}$ is smooth.
\end{proof}

We can also consider a natural generalization to compact subsets of $\mathbb{C}
^{n}$.

\begin{definition}
Suppose that $D,\widetilde{D}$ are closed subsets of $\mathbb{C}^{n}$. Set $
z_{0}=1$. A \emph{degree }$1$ map on $D$ is a function from $D$ to $\mathbb{C
}^{n}$ of the form
\begin{equation*}
\left( z_{1},\ldots ,z_{n}\right) \mapsto \left( f_{1}\left( z_{1},\ldots
,z_{n}\right) ,\ldots ,f_{n}\left( z_{1},\ldots ,z_{n}\right) \right)
\end{equation*}
where for $1\leq k\leq n$ 
\begin{equation*}
f_{k}\left( z_{1},\ldots ,z_{n}\right) =\sum_{i,j=0}^{n}\beta _{ij}^{\left(
k\right) }z_{i}\bar{z}_{j}
\end{equation*}
for some $\beta _{ij}^{\left( k\right) }\in \mathbb{C}$ such that moreover
for every $1\leq k,m\leq n$,
\begin{equation*}
f_{k}\left( z_{1},\ldots ,z_{n}\right) \overline{f_{m}\left( z_{1},\ldots
,z_{n}\right) }=\sum_{i,j=0}^{n}\beta _{ij}^{\left( k,m\right) }z_{i}\bar{z}
_{j}
\end{equation*}
for some $\beta _{ij}^{\left( k,m\right) }\in \mathbb{C}$. A\emph{\ degree }$
1$\emph{\ homeomorphism} from $D$ to $\widetilde{D}$ is a homeomorphism $
\varphi :D\rightarrow \widetilde{D}$ such that $\varphi $ and $\varphi ^{-1}$
are degree $1$ maps. The closed subsets $D$ and $\widetilde{D}$ of $\mathbb{C
}^{n}$ are \emph{degree }$1$ \emph{homeomorphic }if there is a degree $1$
homeomorphism from $D$ to $\widetilde{D}$.
\end{definition}

Again this is an equivalence relation for compact subsets of $\mathbb{C}^{n}$
.

\begin{theorem}
\label{Theorem: degree 1 dimension n}The relation of degree $1$
homeomorphism of subsets of $\mathbb{C}^{n}$ is smooth.
\end{theorem}
\begin{proof} We work along a similar path as the proof of 
Theorem \ref{Theorem: degree 1 dimension 1}, by considering tuples $\left(
V_{1},\ldots ,V_{n}\right) $ of pairwise commuting bounded normal linear
operators on $H$ and the operator system $X$ generated by $V_{1},\ldots
,V_{n}$ and $V_{i}V_{j}^{\ast }$ for $1\leq i,j\leq n$. It is a consequence
of Lemma \ref{Lemma: envelope} that the C*-envelope of $X$ can be identified
with $C^{\ast }\left( V_{1},,\ldots ,V_{n}\right) $. Therefore the proof of
Lemma \ref{Lemma: iso osy normal} shows that two such operator systems are
completely order isomorphic if and only if their spectra are degree $1$
homeomorphic. Working in the parametrization $\Xi $ one can easily see that
there is a Borel map assigning to a compact subset $D$ of $\mathbb{C}^{n}$
the operator system generated by $V_{1},\ldots ,V_{n}$ and $V_{i}V_{j}^{\ast
}$ for $1\leq i,j\leq n$, where $V_{1},\ldots ,V_{n}\in C(D)$ are the
coordinate projections. The proof is concluded by observing that $C(D)$ is
generated by $V_{1},\ldots ,V_{n}$ as a C*-algebra by the Stone-Weierstrass
theorem.
\end{proof}

\section{Smooth classification of structures\label{Section:structures}}

In this section we isolate the model-theoretic content of Theorem~\ref
{Theorem: iso fg osy} and show that the isometric classification of any
class of structures that are compact or at least locally compact will be
smooth. We use the logic for metric structures as defined in \cite[Section
2.1]{farah_model_?}. In particular we consider possibly unbounded metric
structures, endowed with suitable \emph{domains of quantification}$\emph{s}$
. The allowed quantifiers are $\sup_{x\in D}$ and $\inf_{x\in D}$ for some
domain of quantification $D$. The more usual framework of logic for metric
structures from \cite{ben_yaacov_model_2008} only considers structures with
metric bounded by $1$ where the only domain of quantification is the whole structure.

Let $\mathcal{L}$ be a countable language in this logic, and enumerate its
domains of quantification $\left( D_{n}\right) _{n\in \mathbb{N}}$. We can
suppose that $\mathcal{L}$ contains only relation symbols; otherwise we
replace each function symbol with a relation symbol to be interpreted as the
graph of the function. We can also suppose that $\mathcal{L}$ has just one
sort; the multi-sorted version of our result may be easily obtained by a
suitable adaptation of the argument.

We work in the parametrization for separable $\mathcal{L}$-structures
similar to the one considered in \cite{ben_yaacov_lopez-escobar_2014}. Write 
$\mathbb{N}$ as an increasing union $\bigcup_{n}Q_{n}$ of infinite sets such
that $Q_{n+1}\setminus Q_{n}$ is infinite for every $n\in \mathbb{N}$. An $
\mathcal{L}$-structure is regarded as an element of $\prod_{B}\mathbb{R}
^{\left\vert B\right\vert }$ where $B$ varies over all the relation symbols
in $\mathcal{L}$ (including the symbol for the metric $d$) and $\left\vert
B\right\vert $ is the arity of $B$. Any element $f=\left( f_{B}\right) _{B}$
of $\prod_{B}\mathbb{R}^{\left\vert B\right\vert }$ with the right uniform
continuity moduli codes and $\mathcal{L}$-structure $M$ that has as support
the completion of $\mathbb{N}$ with respect to the metric $f_{d}$ on $
\mathbb{N}$. The interpretation of the domain of quantification $D_{n}$ is
the closure inside $M$ of $Q_{n}$. Finally the interpretation of a relation
symbol $B$ is obtained by considering the unique (uniformly) continuous
extension of $f_{B}$ to $M^{\left\vert B\right\vert }$. As observed in \cite
{ben_yaacov_lopez-escobar_2014} the set of tuples $f_{B}\in \mathbb{R}^{B}$
that code an $\mathcal{L}$-structure is a Borel subset of $\prod_{B}\mathbb{R
}^{\left\vert B\right\vert }$ and hence a standard Borel space.

\begin{definition}
An $\mathcal{L}$-structure $M$ is \emph{proper} if the interpretation $D^{M}$
of every domain of quantification is compact.
\end{definition}

It is easy to check that the set $\mathrm{Mod}_{p}\left( \mathcal{L}\right) $
of codes for proper $\mathcal{L}$-structures is a Borel. In fact $\left(
f_{B}\right) \in \mathrm{Mod}_{p}\left( \mathcal{L}\right) $ if and only if $
\left( f_{B}\right) $ is a code for an $\mathcal{L}$-structure and moreover $
\forall n,k\in \mathbb{N}$ $\exists m\in \mathbb{N}$ $\exists x_{1},\ldots
,x_{m}\in Q_{n}$ such that $\forall x\in Q_{n}$ $\exists i\leq n$ such that $
d_{B}\left( x,x_{i}\right) <2^{-k}$.

Following \cite[Definition 1.1]{ben_yaacov_fraisse_2014} we say that a
function $f$ from a metric space $X$ to $\left[ 0,+\infty \right] $ is Kat
\v{e}tov if 
\begin{equation*}
\left\vert f(x)-f(y)\right\vert \leq d\left( x,y\right) \leq f(x)+f(y)
\end{equation*}
for every $x,y\in X$. A function $\psi :X\times Y\rightarrow \left[
0,+\infty \right] $ is an \emph{approximate isometry }from $X$ to $Y$ if it
is separately Kat\v{e}tov in each argument. Equivalently an approximate
isometry from $X$ to $Y$ is a code a metric on the disjoint union of $X$ and 
$Y$ that extends the given metrics on $X$ and $Y$ obtained by setting $
\widehat{d}\left( x,y\right) =\psi \left( x,y\right) $ for $x\in X$ and $
y\in Y$. If $\psi $ is an approximate isometry from $X$ to $Y$ we write $
\psi :X\leadsto Y$. If $X_{0}$ and $Y_{0}$ are subspaces of $X$ and $Y$ and $
\psi :X\leadsto Y$, then one can consider the restriction-truncation $\psi
:X_{0}\leadsto Y_{0}$ which is just the restriction of $\psi $ to $
X_{0}\times Y_{0}$. An approximate isometry is $\varepsilon $-bijection if
for every $r>\varepsilon $ and every $x\in X$ there is $y\in Y$ such that $
\psi \left( x,y\right) <r$ and for every $y\in Y$ there is $x\in X$ such
that $\psi \left( x,y\right) <r$. In the following we will use other
notations and conventions from \cite[Definition 1.1]{ben_yaacov_fraisse_2014}
regarding approximate isometries. Suppose now that $M,N$ are $\mathcal{L}$
-structures, $\psi :M\leadsto N$, and $B\in \mathcal{L}$ is an $n$-ary
relation symbol. Identify the interpretation $B^{M}$ with its graph, and
regard it as a metric space endowed with the $\max $-metric. Define the
approximate isometry $\psi ^{B}:B^{M}\leadsto B^{N}$ by
\begin{equation*}
\psi ^{B}\left( \bar{x},B(\bar{x}),\bar{y},B(\bar{y})\right) =\max \left\{
\psi \left( x_{1},y_{1}\right) ,\ldots ,\psi \left( x_{n},y_{n}\right)
,\left\vert B(\bar{x})-B(\bar{y})\right\vert \right\} \text{.}
\end{equation*}

\begin{definition}
Suppose that $M,N$ are $\mathcal{L}$-structures. Fix $\varepsilon >0$, $k\in 
\mathbb{N}$, a domain of quantification $D$ in $\mathcal{L}$ and a subset $
\mathcal{L}_{0}$ of $\mathcal{L}$. An $\left( \mathcal{L}_{0},\varepsilon
\right) $\emph{-approximate isomorphism} from $M$ to $N$ is an approximate
isometry $M\leadsto N$ such that $\psi ^{B}$ is an $\varepsilon $-isometry
for every $B\in \mathcal{L}_{0}$. A $\left( D,\mathcal{L}_{0},\varepsilon
\right) $-approximate isomorphism from $M$ to $N$ is an approximate isometry 
$\psi :M\leadsto N$ such that the restriction-truncation $\psi
:D_{k}(M)\rightarrow D_{k}(N)$ is an $\left( \mathcal{L}_{0},\varepsilon
\right) $-approximate isomorphism from $D(M)$ to $D(N)$.
\end{definition}

Write the language $\mathcal{L}$ as a countable increasing union of finite
languages $\mathcal{L}_{k}$ for $k\in \mathbb{N}$ such that the metric
symbol $d$ belongs to $\mathcal{L}_{1}$. If $k\in \mathbb{N}$ and $M,N$ are $
\mathcal{L}$-structures define $d_{k}\left( M,N\right) $ to be the infimum
of $\varepsilon >0$ such that there is a $\left( D_{k},\mathcal{L}
_{k},\varepsilon \right) $-approximate isomorphism from $M$ to $N$. The
boundedness requirement on the values of the interpretations of relation
symbols in $\mathcal{L}$-structures shows that $d_{k}\left( M,N\right) $ is
finite. Define then the Gromov--Hausdorff distance
\begin{equation*}
d_{GH}\left( M,N\right) =\sum_{k\in \mathbb{N}}2^{-k}d\left( M,N\right) 
\text{.}
\end{equation*}

Recall the definition of formula from \cite[Section 2.4]{farah_model_?}. A
formula is \emph{universal} if it only uses the quantifier $\sup $.
Similarly a formula is \emph{existential }if it only uses the quantifier $
\inf $. For every $n\in \mathbb{N}$ let us fix a uniformly dense countable
set $\mathcal{F}_{n}$ of functions from the interval $\left[ -n,n\right] $
to itself. A formula is \emph{restricted }if it only uses connectives from $
\bigcup_{n}\mathcal{F}_{n}$. Observe that there are only countably many
restricted $\mathcal{L}$-formulas. (Recall that we are assuming $\mathcal{L}$
to be countable.)

\begin{proposition}
\label{Proposition: isomorphism and theory}Suppose that $M$ and $N$ are two
proper $\mathcal{L}$-structures. The following statements are equivalent:

\begin{enumerate}
\item $\varphi ^{M}=\varphi ^{N}$ for every restricted universal sentence $
\varphi $;

\item $\varphi ^{M}=\varphi ^{N}$ for every universal sentence $\varphi $;

\item $\varphi ^{M}=\varphi ^{N}$ for every existential sentence $\varphi $;

\item $d_{GH}\left( M,N\right) =0$;

\item $M$ and $N$ are isomorphic;

\item $M$ and $N$ are bi-embeddable.
\end{enumerate}
\end{proposition}

\begin{proof}
\noindent (1)$\implies $(2): This follows from an easy approximation
argument, using the uniform density of $\mathcal{F}_{n}$ in the space of
continuous functions from the interval $\left[ -n,n\right] $ to itself.

\noindent (2)$\implies $(3): Use the following standard trick. If $\varphi $
is an existential sentence then there is a large enough $N\in \mathbb{N}$
such that $N-\varphi $ is semantically equivalent to a universal sentence.
Moreover $\left( N-\varphi \right) ^{M}=N-\varphi ^{M}$ for every $\mathcal{L
}$-structure $M$.

\noindent (3)$\implies $(4):  This follows easily from the definition of the
Gromov--Hausdorff distance.

\noindent (4)$\implies $(5):  Fix countable subsets $M_{0}$ and $N_{0}$ of $M$
and $N$ such that $M_{0}\cap D_{k}(M)$ is dense in $D_{k}(M)$ and $N_{0}\cap
D_{k}(N)$ is dense in $D_{k}(N)$ for every $k\in \mathbb{N}$. Then for every 
$k\in \mathbb{N}$ there is a $\left( D_{k},\mathcal{L}_{k},2^{-k}\right) $
-approximate isomorphism $\psi $ from $M$ to $N$. This allows one to define
functions $f_{k}:M\rightarrow N$ such that $f_{k}\left[ D_{k}(M)\right]
\subset D_{k}(N)$ and for every $n$-ary relation $B$ in $\mathcal{L}_{k}$,
every $\bar{x}\in \left( M_{0}\cap D_{k}(M)\right) ^{n}$
\begin{equation*}
\psi ^{B}\left( \bar{x},f_{k}(\bar{x})\right) <2^{-k}
\end{equation*}
and for every $y\in N_{0}\cap D_{k}(M)$ there is $z\in M_{0}\cap D_{k}(M)$
such that
\begin{equation*}
d\left( f_{k}(z),y\right) <2^{-k}\text{.}
\end{equation*}
By compactness of $D_{k}(N)$, after passing to a subsequence we can assume
that the sequence $\left( f_{k}(x)\right) _{k\in \mathbb{N}}$ converges in $
N $ for every $x\in M_{0}$. Defining $f_{\infty }(x)$ to be the limit of the
sequence $\left( f_{k}(x)\right) _{k\in \mathbb{N}}$ defines an isometry $
f_{\infty }$ from $M_{0}$ onto a dense subset of $N_{0}$ that preserves all
the relations $\mathcal{L}$. The unique isometric extension of $f_{\infty }$
to the whole $M$ defines an isomorphism from $M$ onto $N$.

\noindent 5)$\implies $(6)$\implies $(1) These are clear by definition.
\qedhere
\end{proof}

Proposition \ref{Proposition: isomorphism and theory} can be regarded as the
metric analogue of the classical fact that a finite discrete structure is
classified by its first order theory.

\begin{corollary}
\label{Corollary: smooth isomorphism}The isomorphism relation for proper $
\mathcal{L}$-structures is smooth.
\end{corollary}

\begin{proof}
It can be easily shown by induction on the complexity of a formula $\varphi
\left( x_{1},\ldots ,x_{n}\right) $ that the interpretation function 
\begin{equation*}
\left( M,a_{1},\ldots ,a_{n}\right) \mapsto \varphi ^{M}\left( a_{1},\ldots
,a_{n}\right)
\end{equation*}
is a Borel function from $\mathrm{Mod}\left( \mathcal{L}\right) \times 
\mathbb{N}^{n}$ to $\mathbb{R}$. This has been shown in the particular case
of C*-algebras in \cite[Proposition 5.1]{farah_descriptive_2012}. Fix an
enumeration $\left( \varphi _{n}\right) _{n\in \mathbb{N}}$ of all
restricted universal sentences. By Proposition \ref{Proposition: isomorphism
and theory} the function
\begin{equation*}
M\mapsto \left( \varphi _{n}^{M}\right) _{n\in \mathbb{N}}
\end{equation*}
from $\mathrm{Mod}_{p}\left( \mathcal{L}\right) $ to $\mathbb{R}^{\mathbb{N}
} $ is a Borel reduction from the relation of isomorphism of proper $
\mathcal{L}$-structures to equality of sequences of real numbers.
\end{proof}

An alternative way to obtain Corollary \ref{Corollary: smooth isomorphism}
is to define the complete separable metric space $X_{\mathcal{L}}$ of
isomorphism classes of proper $\mathcal{L}$-structures endowed with the
Gromov--Hausdorff metric $d_{GH}$. The equivalence of (4) and (5) in
Proposition \ref{Proposition: isomorphism and theory} shows that $\left( X_{
\mathcal{L}},d_{GH}\right) $ is indeed a metric space. After coding any
function $f$ by the relation expressing the distance from the graph of $f$,
we can assume that $\mathcal{L}$ contains only relation symbols. Thus
separability follows by considering the countable dense collection of finite 
$\mathcal{L}$-structures with rational-valued relations. Finally
completeness follows from the standard argument allowing to build the
Gromov--Hausdorff limit of a Cauchy sequence; see for example \cite[
Proposition 43]{petersen_riemannian_2006}. The proof is concluded observing
that the function that assigns to a code for an $\mathcal{L}$-structure $
M\in \mathrm{Mod}_{p}\left( \mathcal{L}\right) $ its isomorphism class $
\left[ M\right] \in X_{\mathcal{L}}$ is a Borel reduction from isomorphism
to equality.

In view of the Choi--Effros abstract characterization \cite[Theorem 13.1]
{paulsen_completely_2002}, operator systems can be described as structures
in a suitable language $\mathcal{L}_{OSy}$ in \cite[Appendix B]
{goldbring_kirchberg_2014}; see also \cite[Section 3.3]
{elliott_isomorphism_2013}. In this case the domains of quantifications are
norm balls of matrix amplifications. Therefore the following is an immediate
consequence of Proposition \ref{Proposition: isomorphism and theory}.

\begin{corollary}
Suppose that $X$ and $Y$ are finitely generated operator systems. Then $X$
and $Y$ are complete order isomorphic if and only if $\varphi ^{X}=\varphi
^{Y}$ for every universal $\mathcal{L}_{OSy}$-sentence. In particular the
complete order isomorphism for finitely generated operator systems is smooth.
\end{corollary}

The same result equally applies to finitely generated operator spaces. These
can also be regarded as structures in a language $\mathcal{L}_{OSp}$ \cite[
Appendix B]{goldbring_kirchberg_2014}, using Ruan's abstract
characterization \cite[Theorem 13.4]{paulsen_completely_2002}.

\begin{corollary}
Suppose that $X$ and $Y$ are finitely generated operator spaces. Then $X$
and $Y$ are completely isometric if and only if $\varphi ^{X}=\varphi ^{Y}$
for every universal $\mathcal{L}_{OSy}$-sentence. In particular the complete
isometry relation for finitely generated operator spaces is smooth.
\end{corollary}

Finally one can consider finitely generated $\emph{unital}$ operator spaces.
The abstract characterization provided by \cite[Theorem 1.1]
{blecher_metric_2011} shows that unital operator spaces can be regarded as $
\mathcal{L}_{uOSp}$-structures in a suitable language $\mathcal{L}_{uOSp}$.

\begin{corollary}
Suppose that $X$ and $Y$ are finitely generated operator spaces. Then $X$
and $Y$ are unitally completely isometric if and only if $\varphi
^{X}=\varphi ^{Y}$ for every universal $\mathcal{L}_{uOSy}$-sentence. In
particular the unital complete isometry relation for finitely generated
unital operator spaces is smooth.
\end{corollary}

\appendix

\section{Equivalence of parametrizations of operator systems\label
{Appendix:equivalenceOSy}}

In this Appendix we show that the parametrizations of operator systems $
\Gamma $, $\Xi $, and $\widehat{\Xi }$ are equivalent. Furthermore we show
that $\Gamma _{N}$ and $\widehat{\Gamma }_{N}$ provide weakly equivalent
parametrizations of $N$-dimensional operator systems.

The argument of the following lemma is analogous to the one of \cite[Lemma
2.4]{farah_turbulence_2014}. The full proof is presented for the convenience
of the reader.

\begin{lemma}
\label{Lemma: Borel injection}Suppose that $X$ is a standard Borel space,
and $Y$ is any of the spaces $\Gamma $, $\Xi $ or $\widehat{\Xi }$. Let $f$
be a Borel function from $X$ to $Y$. Then there is a Borel injection $
\widetilde{f}$ from $X$ to $Y$ such that $\mathcal{OS}y(\widetilde{f}
(x))\cong \mathcal{OS}y\left( f(x)\right) $ for every $x\in X$.
\end{lemma}

\begin{proof}
Consider the case where $Y=\Gamma $. Without loss of generality we can assume
that $X$ is the standard Borel space of infinite subsets of $\mathbb{N}$
regarded as a subset of $2^{\mathbb{N}}$. Denote by $I$ the identity
operator on $H$. Define the function $\widetilde{f}:X\rightarrow \Gamma $ by
setting
\begin{equation*}
\widetilde{f}(A)_{k}=\left\{ 
\begin{array}{ll}
nI & \text{if }k=2^{n}\text{ for some }n\in A\text{,} \\ 
f(A)_{n} & \text{if }k=3^{n}\text{ for some }n\in \mathbb{N}\text{,} \\ 
0 & \text{otherwise.}
\end{array}
\right.
\end{equation*}
Then the function $\widetilde{f}$ is a Borel injection and $\mathcal{OS}y(
\widetilde{f}(A))\cong \mathcal{OS}y\left( f(A)\right) $ for every $A\in X$.

Consider now the case when $Y=\widehat{\Xi }$. Without loss of generality we
can assume that $X$ is the standard Borel space of infinite subsets of the
set of \emph{even} natural numbers. Observe that the group $S_{\infty }$ of
permutations of $\mathbb{N}$ has a natural action on $\widehat{\Xi }$.
Explicitly if $\left( f,g,h,(C_{\xi })_{n\in \mathbb{N}},e_{\xi }\right) \in 
\widehat{\Xi }$ and $\sigma \in S_{\infty }$ then $\sigma \cdot \xi $ is the
element of $\widehat{\Xi }$ obtained replacing $f$ with the function $\left(
n,m\right) \mapsto \sigma \left( f\left( \sigma ^{-1}\left( n\right) ,\sigma
^{-1}\left( m\right) \right) \right) $ and similarly with the other entries
of $\xi $. It is clear that $\mathcal{OS}y(\xi )\cong \mathcal{OS}y\left(
\sigma \cdot \xi \right) $ for any $\xi \in \widehat{\Xi }$ and $\sigma \in
S_{\infty }$. Given $\xi \in \widehat{\Xi }$ and $A\in X$ one can find in a
Borel way a permutation $\sigma _{\xi ,A}$ of $\mathbb{N}$ such that
\begin{equation*}
A=\left\{ 2^{n}\cdot _{\sigma \cdot \xi }1:n\in \mathbb{N}\right\} \text{.}
\end{equation*}
We can then define the Borel injection $\widetilde{f}:X\rightarrow \widehat{
\Xi }$ by
\begin{equation*}
\widetilde{f}(A)=\sigma _{f(A),A}\cdot f(A)\text{.}
\end{equation*}

Finally suppose that $Y=\Xi $. We can assume without loss of generality that 
$X$ is the space of positive real numbers. Fix $x\in X$ and define $\delta
=f(x)$. Consider the least $n_{0}\in \mathbb{N}$ such that $\delta
_{1}\left( X_{n_{0}}\right) \neq 0$. Let $X$ be an operator system and $
\gamma $ be a dense sequence in $X$ such that, for every $\left[ q_{ij}
\right] \in M_{n}(\mathcal{V})$, 
\begin{equation*}
\delta _{n}\left( \left[ q_{ij}\right] \right) =\left\Vert \left[
q_{ij}(\gamma )\right] \right\Vert _{M_{n}(A)}\text{.}
\end{equation*}
Let $\widetilde{\gamma }$ to be the sequence in $X$ defined by
\begin{equation*}
\widetilde{\gamma }_{n}=\left\{ 
\begin{array}{ll}
\frac{x}{\left\Vert \gamma _{i}\right\Vert _{A}}\gamma _{i} & \text{if }
i=n_{0}\text{,} \\ 
\gamma _{i} & \text{otherwise.}
\end{array}
\right.
\end{equation*}
Define 
\begin{equation*}
\widetilde{f}(x)_{n}\left( \left[ q_{ij}\right] \right) =\left\Vert \left[
q_{ij}(\widetilde{\gamma })\right] \right\Vert _{M_{n}(A)}\text{.}
\end{equation*}
Observe that $\widetilde{f}$ is well defined and injective. Note also that $
\widetilde{f}(x)_{n}=\delta ^{\prime }$ if and only if there are $\gamma
,\gamma ^{\prime }\in \Gamma $ and $\delta \in \Xi $ such that $\left\Vert
p(\gamma )\right\Vert =\delta \left( p\right) $ and for every $p\in \mathcal{
V}$, $n\in \mathbb{N}$, and $q_{ij}\in M_{n}(\mathcal{V})$, and $\delta
^{\prime }([q_{ij}^{\prime }])=\delta \left( \left[ q_{ij}\right] \right) $.
Here $q_{ij}^{\prime }$ is the polynomial obtained from $q_{ij}$ by
replacing any occurrence of $X_{n_{0}}$ with $\left( \delta
_{1}(X_{n_{0}})/x\right) X_{n_{0}}$, while $n_{0}$ is defined as above to be
the least natural number such that $\delta _{1}\left( X_{n_{0}}\right) $ is
nonzero. This observation gives an analytic definition for the function $
\widetilde{f}$. The classical principle that a function with analytic graph
is Borel \cite[Theorem 14.12]{kechris_classical_1995} concludes the proof.
\end{proof}

\begin{proposition}
\label{Proposition: equipar osy}The parametrizations of operator systems $
\Gamma $, $\Xi $, and $\widehat{\Xi }$ are equivalent.
\end{proposition}

\begin{proof}
In view of Lemma \ref{Lemma: Borel injection} it is enough to show that the
parametrizations $\Gamma $, $\Xi $, and $\widehat{\Xi }$ are weakly
equivalent. Suppose that $\gamma \in \Gamma $. Define the matrix ordered $
\mathbb{Q}(i)$-$\ast $-vector space structure $\xi _{\gamma }$ on $\mathbb{N}
$
\begin{align*}
n+_{\xi _{\gamma }}m=k& \iff \mathfrak{p}_{n}(\gamma )+\mathfrak{p}
_{m}(\gamma )=\mathfrak{p}_{k}(\gamma )\text{,} \\
q\cdot _{\xi _{\gamma }}n=k& \iff q\mathfrak{p}_{n}(\gamma )=\mathfrak{p}
_{k}(\gamma )\text{,} \\
n^{\ast _{\xi _{\gamma }}}=k& \iff \mathfrak{p}_{n}(\gamma )^{\ast }=
\mathfrak{p}_{k}(\gamma )\text{.}
\end{align*}
The positive cones are defined by
\begin{equation*}
\left[ m_{ij}\right] \in C_{\xi ,n}\iff \lbrack \mathfrak{p}_{m_{ij}}(\gamma
)]\geq 0\text{.}
\end{equation*}
Finally $e_{\xi _{\gamma }}=1$ is an Archimedean matrix order unit. This
defines a Borel function $\gamma \mapsto \xi _{\gamma }$ from $\Gamma $ to $
\widehat{\Xi }$ such that $\mathcal{OS}y(\gamma )\cong \mathcal{OS}y\left(
\xi _{\gamma }\right) $.

Now suppose that $\xi \in \widehat{\Xi }$. If $p\in \mathcal{V}$ define $
p^{\xi }$ to be the element of $\mathbb{N}$ obtained by evaluating $p$ in
the $\mathbb{Q}(i)$-$\ast $-vector space structure on $\mathbb{N}$ defined
by $\xi $ after by replacing $X_{i}$ with $i$ and replacing the constant $c$
by $c\cdot _{\xi }e_{\xi }$. Denote by $rI_{n}^{\xi }$ the $n\times n$
matrix with positive integer coefficients having $r\cdot _{\xi }e_{\xi }$ in
the diagonal entries and $0\cdot _{\xi }e_{\xi }$ elsewhere. Define $\delta
_{\xi }\in \Xi $ by setting, for $P\in M_{n}(\mathcal{V})$ and $r\in \mathbb{
Q}_{+}$,
\begin{equation*}
\delta _{\xi ,n}\left( P\right) <r\iff 
\begin{bmatrix}
rI_{n}^{\xi } & P \\ 
P^{\ast } & rI_{n}^{\xi }
\end{bmatrix}
\in C_{\xi ,2n}\text{.}
\end{equation*}
By \cite[Proposition 13.3]{paulsen_completely_2002} this defines a Borel map 
$\xi \mapsto \delta _{\xi }$ from $\widehat{\Xi }$ to $\Xi $ such that $
\mathcal{OS}y(\xi )\cong \mathcal{OS}y\left( \delta _{\xi }\right) $.

To conclude the proof it is now enough to describe a Borel function $\delta
\mapsto \gamma _{\delta }$ from $\Xi $ to $\Gamma $ such that $\mathcal{OS}
y(\gamma _{\delta })\cong \mathcal{OS}y(\delta )$. For $\delta \in \Xi $ and 
$k\in \mathbb{N}$ define $P_{k}(\delta )$ the set of $\phi \in M_{k}(\mathbb{
C})^{\mathcal{V}}$ such that

\begin{itemize}
\item $\left\Vert \left[ \phi \left( p_{ij}\right) \right] \right\Vert
_{M_{kn}(\mathbb{C})}\leq \delta _{n}\left( \left[ p_{ij}\right] \right) $
for every $n\in \mathbb{N}$ and $\left[ p_{ij}\right] \in M_{n}(\mathcal{V})$
, and

\item $\phi \left( \mathfrak{p}_{1}\right) $ is the identity matrix.
\end{itemize}

(Recall that $\mathfrak{p}_{1}$ is the constant polynomial $1$.) Then $
P_{k}(\delta )$ is a compact subset of $M_{k}(\mathbb{C})^{\mathcal{V}}$
with the product topology. Moreover the relation
\begin{equation*}
\left\{ \left( \delta ,P_{k}(\delta )\right) \in \Xi \times M_{k}(\mathbb{C}
)^{\mathcal{V}}:\phi \in P_{k}(\delta )\right\}
\end{equation*}
is Borel. It follows from \cite[Theorem 28.8]{kechris_classical_1995} that
the function
\begin{equation*}
\delta \mapsto P_{k}(\delta )
\end{equation*}
is a Borel function into the Polish space $K(M_{k}(\mathbb{C})^{\mathcal{V}
}) $ of compact subsets of $M_{k}(\mathbb{C})^{\mathcal{V}}$. Consider the
Borel set $A_{n}$ of tuples $\left( \delta ,\varepsilon ,\left[ m_{ij}\right]
,k,\phi \right) $ where

\begin{itemize}
\item $\delta \in \Xi $,

\item $\varepsilon \in \mathbb{Q}_{+}$,

\item $\left[ p_{ij}\right] \in M_{n}(\mathcal{V})$,

\item $k\in \left\{ 1,2,\ldots ,n\right\} $, and

\item $\phi \in P_{k}(\delta )$ is such that $\left\Vert \left[ \phi \left(
p_{ij}\right) \right] \right\Vert _{kn}\geq \delta _{n}\left( \left[ p_{ij}
\right] \right) -\varepsilon $.
\end{itemize}

The proof of the Choi-Effros abstract characterization of operator
systems  \cite[page 179]{paulsen_completely_2002} shows
that for every $\left( \delta ,\varepsilon ,\left[ p_{ij}\right] \right) \in
\Xi \times \mathbb{Q}_{+}\times M_{n}(\mathcal{V})$ the corresponding
section of $A_{n}$ is (compact and) nonempty. Therefore by \cite[Theorem 28.8
]{kechris_classical_1995} there is a Borel function 
\begin{equation*}
\left( \delta ,\varepsilon ,\left[ p_{ij}\right] \right) \mapsto \phi
_{\delta ,\varepsilon ,[p_{ij}]}\in P_{k_{\delta ,\varepsilon
,[p_{ij}]}}(\delta )
\end{equation*}
such that $\left\Vert \left[ \phi \left( p_{ij}\right) \right] \right\Vert
_{k_{\delta ,\varepsilon ,[p_{ij}]}n}\geq \delta _{n}\left( \left[ p_{ij}
\right] \right) -\varepsilon $. Denote by $\mathcal{M}$ the set of $n\times
n $ matrices $\left[ p_{ij}\right] $ with entries in $\mathcal{V}$ where $n$
varies in $\mathbb{N}$. Denote by $H$ the separable Hilbert space with
orthonormal basis $\left( e_{[p_{ij}],\varepsilon ,\alpha }\right) $ indexed
by $\mathcal{M}\times \mathbb{Q}_{+}\times \mathbb{N}$. For every $k\in 
\mathbb{N}$ denote by $(b_{n,\alpha })_{\alpha \leq n}$ the canonical basis
of $\mathbb{C}^{k}$. For $\delta \in \widehat{\Xi }$ and $n\in \mathbb{N}$
denote by $\gamma _{\delta ,n}$ the element of $B(H)$ defined by setting 
\begin{equation*}
\left\langle \gamma _{\delta ,n}e_{[p_{ij}],\varepsilon ,\alpha
},e_{[q_{ij}],\varepsilon ,\beta }\right\rangle =\left\langle \phi _{\delta
,\varepsilon ,[p_{ij}]}\left( X_{n}\right) e_{k_{\delta ,\varepsilon
,[p_{ij}]},\alpha },b_{k_{\delta ,\varepsilon ,[p_{ij}]},\beta }\right\rangle
\end{equation*}
if $p_{ij}=q_{ij}$ and $\alpha ,\beta \leq k_{\delta ,\varepsilon ,[p_{ij}]}$
, and zero otherwise. The construction ensures that the map $p\mapsto
p(\gamma )$ extends to a complete isometry from the operator system $
\mathcal{OS}y(\delta )$ coded by $\delta $ and the operator system $\mathcal{
OS}y(\gamma _{\delta })$ coded by the sequence $\gamma _{\delta }=\left(
\gamma _{\delta ,n}\right) _{n\in \mathbb{N}}$. Observing that the function $
\delta \mapsto \gamma _{\delta }$ is Borel concludes the proof.
\end{proof}

We will now verify that the sets $\Gamma _{N}$ and $\widehat{\Gamma }_{N}$ provide
weakly equivalent parametrizations of the $N$-dimensional operator systems.
Recall that $\Gamma _{N}$ is the set of $\gamma \in \Gamma $ such that $
\mathcal{OS}(\gamma )$ has dimension $N$. Similarly $\widehat{\Gamma }_{N}$
is the set of linearly independent tuples $\left( x_{1},\ldots ,x_{N}\right) 
$ in $B\left( H\right) $ such that $\mathrm{span}\left\{ x_{1},\ldots
,x_{N}\right\} $ is an operator system. Denote as before by $\mathcal{V}_{
\mathbb{C}}$ the complex $\ast $-vector space of noncommutative $\ast $
-polynomials with coefficients from $\mathbb{C}$. Endow $\mathcal{V}_{
\mathbb{C}}$ with the norm
\begin{equation*}
\left\Vert \sum_{i\leq n}a_{i}X_{i}+\sum_{i\leq n}b_{i}X_{i}^{\ast
}+c\right\Vert =\sum_{i\leq n}\left\vert a_{i}\right\vert +\sum_{i\leq
n}\left\vert b_{i}\right\vert +\left\vert c\right\vert \text{.}
\end{equation*}
Let us show that the set $\Gamma _{\leq N}$ of $\gamma \in \Gamma $ such
that $\mathcal{OS}y(\gamma )$ has dimension at most $N$ is Borel. To this
purpose it is enough to show that it is both analytic and coanalytic \cite[
Theorem 14.1]{kechris_classical_1995}. Observe that on one hand $\gamma \in
\Gamma _{\leq N}$ if and only if there are $p_{1},\ldots ,p_{N}\in \mathcal{V
}_{\mathbb{C}}$ such that for every $n\in \mathbb{N}$ and $q\in \mathcal{V}$
there is $r\in \mathcal{V}$ such that
\begin{equation*}
\left\Vert r\left( p_{1}(\gamma ),\ldots ,p_{N}(\gamma )\right) -q(\gamma
)\right\Vert <\frac{1}{n}\text{.}
\end{equation*}
On the other hand $\gamma \in \Gamma _{\leq N}$ if and only if for every $
p_{1},\ldots ,p_{N+1}\in \mathcal{V}_{\mathbb{C}}$ there are $\lambda
_{1},\ldots ,\lambda _{N+1}\in \mathbb{Q}(i)$ such that $\lambda _{i}\neq 0$
for some $i\in \left\{ 1,2,\ldots ,N+1\right\} $ and
\begin{equation*}
\lambda _{1}p_{1}+\cdots +\lambda _{N+1}p_{N+1}=0\text{.}
\end{equation*}
This shows that $\Gamma _{\leq N}$ is both analytic and coanalytic, and
hence Borel.

\begin{lemma}
\label{Lemma: choose basis}There are Borel functions $\gamma \mapsto
b_{i}^{\gamma }$ from $\Gamma _{N}$ to $\mathcal{V}$ for $i\leq N$ such that 
$\left\{ b_{1}^{\gamma }(\gamma ),\ldots ,b_{N}^{\gamma }(\gamma )\right\} $
is a basis for $\mathcal{OS}y(\gamma )$.
\end{lemma}

\begin{proof}
We show by induction on $k\leq N$ that there are Borel functions $\gamma
\mapsto b_{i}^{\gamma }$ from $\Gamma _{N}$ to $\mathcal{V}$ for $i\leq k$
such that $\left\{ b_{1}^{\gamma }(\gamma ),\ldots ,b_{N}^{\gamma }(\gamma
)\right\} $ is a linearly independent set. For $k=1$ this is immediate.
Suppose that $b_{1}^{\gamma },\ldots ,b_{k}^{\gamma }$ have been defined for 
$k<N$. Define the relation $A$ to be the set of pairs $\left( \gamma
,p\right) \in \Gamma _{N}\times \mathcal{V}$ such that $p(\gamma )\notin 
\mathrm{span}\left\{ b_{1}^{\gamma }(\gamma ),\ldots ,b_{k}^{\gamma }(\gamma
)\right\} $. Since for every $\gamma \in \Gamma _{N}$ the corresponding
section $A_{\gamma }$ is nonempty, there is a Borel function $
b_{k+1}^{\gamma }:\Gamma _{N}\rightarrow \mathcal{V}$ such that $\left(
\gamma ,b_{k+1}^{\gamma }(\gamma )\right) \in A$.
\end{proof}

Denote in the following by $b_{1}^{\gamma },\ldots ,b_{N}^{\gamma }$ the
functions defined in Lemma \ref{Lemma: choose basis}. Observe that the maps $
\gamma \mapsto b_{k}^{\gamma }$ from $\Gamma $ to $B(H)$ are Borel. Denote
as before by $\mathcal{W}_{N}$ the set of polynomials of degree $1$ in the
noncommutative variables $X_{1},\ldots ,X_{N}$ and with coefficients from $
\mathbb{Q}(i)$. Similarly denote by $\mathcal{V}_{N}$ the set of polynomials
of degree $1$ in the noncommutative variables $X_{1},\ldots
,X_{N},X_{1}^{\ast },\ldots ,X_{N}^{\ast }$ and with coefficients from $
\mathbb{Q}(i)$.

\begin{lemma}
\label{Lemma: GAMMA HAT}$\widehat{\Gamma }_{N}$ is a Borel subset of $
B(H)^{N}$.
\end{lemma}

\begin{proof}
Observe that $\bar{x}\in \widehat{\Gamma }_{N}$ if and only if

\begin{itemize}
\item for every $\varepsilon \in \mathbb{Q}_{+}$ there is $p\in \mathcal{V}
_{N}$ such that $\left\Vert p(x)-I\right\Vert <\varepsilon $,

\item for every $p\in \mathcal{V}_{N}$ and $\varepsilon \in \mathbb{Q}_{+}$
there is $q\in \mathcal{W}_{N}$ such that 
\begin{equation*}
\left\Vert p(x)-q(x)\right\Vert <\varepsilon \text{,}
\end{equation*}
and

\item for every $k<N$ there is $\varepsilon \in \mathbb{Q}_{+}$ such that
for every $q\in \mathcal{W}_{k}$ 
\begin{equation*}
\left\Vert q\left( x_{1},\ldots ,x_{k}\right) -x_{k+1}\right\Vert \geq
\varepsilon \text{.}\qedhere
\end{equation*}
\end{itemize}
\end{proof}

The operator system associated with $\left( x_{1},\ldots ,x_{N}\right) \in 
\widehat{\Gamma }_{N}$ is the  span of $\left\{ x_{1},\ldots ,x_{N}\right\} $
. Lemma \ref{Lemma: GAMMA HAT} shows that $\widehat{\Gamma }_{N}$ is a
standard Borel parametrization of $N$-dimensional operator systems. By Lemma 
\ref{Lemma: choose basis} such a parametrization is weakly equivalent to the
parametrization $\Gamma _{N}$.

\section{Equivalence of parametrizations of operator spaces\label
{Appendix:equivalenceOSp}}

In this appendix we show that the parametrizations for operator spaces $\Xi $
, $\widehat{\Xi }$, and $\Gamma $ are equivalent. The proof of the following
lemma is entirely analogous to the proof of Lemma \ref{Lemma: Borel
injection} and \cite[Lemma 2.4]{farah_turbulence_2014}.

\begin{lemma}
\label{Lemma: Borel injection 2}Suppose that $X$ is a standard Borel space,
and $Y$ is any of the space $\Gamma $, $\Xi $, and $\widehat{\Xi }$. If $f$
is a Borel function from $X$ to $Y$, then there is a Borel injection $
\widetilde{f}$ from $X$ to $Y$ such that $\mathcal{OS}p\left( f(x)\right)
\cong \mathcal{OS}p\left( f(x)\right) $ for every $x\in X$.
\end{lemma}

\begin{proposition}
\label{Proposition: equipar osp}The parametrizations of operator spaces $
\Gamma $, $\Xi $, and $\widehat{\Xi }$ are equivalent.
\end{proposition}

\begin{proof}
In view of Lemma \ref{Lemma: Borel injection 2} it is enough to show that $
\Gamma $, $\Xi $, and $\widehat{\Xi }$ are weakly equivalent.
Isomorphism-preserving Borel functions from $\Gamma $ to $\widehat{\Xi }$
and from $\widehat{\Xi }$ to $\Xi $ can be easily defined as in\ Proposition 
\ref{Proposition: equipar osy}. Hence we focus here on constructing an
isomorphism-preserving Borel function from $\Xi $ to $\Gamma $. Observe that
we can identify $\mathcal{V}$ with the $\mathbb{Q}(i)$-$\ast $-vector space $
\mathbb{\mathbb{Q}}(i)\oplus \mathbb{\mathbb{Q}}(i)\oplus \mathcal{W}\oplus 
\overline{\mathcal{W}}$, where $\overline{\mathcal{W}}$ denote the complex
conjugate of the $\mathbb{Q}(i)$-vector spaces $\mathcal{W}$. For
convenience we represent an element of $\mathcal{V}$ as a matrix
\begin{equation*}
\begin{bmatrix}
\lambda & p \\ 
q^{\ast } & \mu
\end{bmatrix}
\end{equation*}
where $\lambda ,\mu \in \mathbb{Q}(i)$ and $p,q\in \mathcal{W}$. Similarly
an $n\times n$ matrix $V$ of elements of $\mathcal{V}$ can be regarded,
after a canonical shuffle, as
\begin{equation*}
\begin{bmatrix}
P & X \\ 
Y^{\ast } & Q
\end{bmatrix}
\end{equation*}
where $P,Q\in M_{n}(\mathbb{\mathbb{Q}}(i))$ and $X,Y\in M_{n}(W)$. We will
adopt these identifications throughout the rest of the proof. Suppose that $
\delta \in \Xi $. Define $C_{n}$ to be the set of
\begin{equation*}
\begin{bmatrix}
P & X \\ 
X^{\ast } & Q
\end{bmatrix}
\end{equation*}
such that $P,Q\in M_{n}(\mathbb{\mathbb{Q}}(i))$ are positive, and
\begin{equation*}
\left\Vert \left( P+\varepsilon I_{n}\right) ^{-1}X\left( Q+\varepsilon
I_{n}\right) ^{-1}\right\Vert \leq 1
\end{equation*}
for every $\varepsilon \in \mathbb{Q}_{+}$. Define then for $V\in M_{n}(
\mathcal{V})$
\begin{equation*}
\widehat{\delta }_{n}(V)<r\iff 
\begin{bmatrix}
I_{n} & V \\ 
V^{\ast } & I_{n}
\end{bmatrix}
\in C_{2n}\text{.}
\end{equation*}
The proof of the abstract characterization of operator spaces due to Ruan 
\cite[Theorem 13.4]{paulsen_completely_2002} shows that 
\begin{equation*}
(\widehat{\delta }_{n})_{n\in \mathbb{N}}\in \prod_{n\in \mathbb{N}}\mathbb{R
}^{M_{n}(V)}
\end{equation*}
is a code for an operator system in the parametrization $\Xi $ for operator
systems. Moreover the function
\begin{equation*}
q\mapsto 
\begin{pmatrix}
0 & q \\ 
0 & 0
\end{pmatrix}
\end{equation*}
induces a complete isometry from the operator space coded by $\delta $ into
the operator system coded by $\widehat{\delta }$. The proof of Proposition 
\ref{Proposition: equipar osy} allows one to assign in a Borel way to $
\widehat{\delta }$ a sequence $\gamma _{\delta }$ in $B(H)$ such that the
function $p\mapsto p(\gamma _{\delta })$ induces a complete isometry from
the operator system coded by $\widehat{\delta }$ onto the operator system
generated by $\gamma _{\delta }$. Therefore the function $q\mapsto q(\gamma
_{\delta })$ induces a complete isometry from the operator space coded by $
\delta $ onto the operator space generated by $\gamma _{\delta }$. The proof
is concluded by observing that the construction above shows that the
assignment $\delta \mapsto \gamma _{\delta }$ is Borel.
\end{proof}


\begin{thebibliography}{10}

\bibitem{arveson_subalgebras_1969}
W.~Arveson.
\newblock Subalgebras of {C}*-algebras.
\newblock {\em Acta Mathematica}, 123(1):141--224, Dec. 1969.

\bibitem{arveson_subalgebras_1972}
W.~Arveson.
\newblock Subalgebras of {C}*-algebras {II}.
\newblock {\em Acta Mathematica}, 128(1):271--308, July 1972.

\bibitem{arveson_noncommutative_2008}
W.~Arveson.
\newblock The noncommutative {C}hoquet boundary.
\newblock {\em Journal of the American Mathematical Society}, 21(4):1065--1084,
  2008.

\bibitem{arveson_noncommutative_2010}
W.~Arveson.
\newblock The noncommutative {C}hoquet boundary {III}: operator systems in
  matrix algebras.
\newblock {\em Mathematica Scandinavica}, 106(2):196{\textendash}210, 2010.


\bibitem{ben_yaacov_fraisse_2014}
I.~Ben~Yaacov.
\newblock {Fra{\"i}ss\'{e} limits of metric structures}.
\newblock {\em Journal of Symbolic Logic}.
\newblock to appear.

\bibitem{ben_yaacov_model_2008}
I.~Ben~Yaacov, A.~Berenstein, C.~W. Henson, and A.~Usvyatsov.
\newblock Model theory for metric structures.
\newblock In {\em Model theory with applications to algebra and analysis. Vol.
  2}, volume 350 of {\em London Mathematical Society Lecture Note Series},
  pages 315--427. Cambridge University Press, 2008.

\bibitem{ben_yaacov_lopez-escobar_2014}
I.~Ben~Yaacov, A.~Nies, and T.~Tsankov.
\newblock {A Lopez-Escobar theorem for continuous logic}.
\newblock {\em {arXiv}:1407.7102}, 2014.

\bibitem{blecher_operator_2004}
D.~P. Blecher and C.~Le~Merdy.
\newblock {\em Operator algebras and their modules---an operator space
  approach}, volume~30 of {\em London Mathematical Society Monographs. New
  Series}.
\newblock Oxford University Press, Oxford, 2004.
\newblock Oxford Science Publications.

\bibitem{blecher_metric_2011}
D.~P. Blecher and M.~Neal.
\newblock Metric characterizations of isometries and of unital operator spaces
  and systems.
\newblock {\em Proceedings of the American Mathematical Society},
  139(3):985--998, 2011.

\bibitem{blecher_metric_2013}
D.~P. Blecher and M.~Neal.
\newblock Metric characterizations {II}.
\newblock {\em Illinois Journal of Mathematics}, 57(1):25--41, 2013.

\bibitem{bodmann_decoherence_2007}
B.~G. Bodmann, D.~W. Kribs, and V.~I. Paulsen.
\newblock Decoherence-insensitive quantum communication by optimal
  {C}*-encoding.
\newblock {\em {IEEE} Transactions on Information Theory}, 53(12):4738--4749,
  Dec. 2007.

\bibitem{bodmann_frame_2007}
B.~G. Bodmann and V.~I. Paulsen.
\newblock Frame paths and error bounds for sigma{\textendash}delta
  quantization.
\newblock {\em Applied and Computational Harmonic Analysis}, 22(2):176--197,
  Mar. 2007.

\bibitem{bodmann_smooth_2007}
B.~G. Bodmann, V.~I. Paulsen, and S.~A. Abdulbaki.
\newblock Smooth frame-path termination for higher order sigma-delta
  quantization.
\newblock {\em Journal of Fourier Analysis and Applications}, 13(3):285--307,
  June 2007.

\bibitem{bollobas_linear_1999}
B.~Bollob\'{a}s.
\newblock {\em Linear analysis}.
\newblock Cambridge University Press, Cambridge, second edition, 1999.
\newblock An introductory course.

\bibitem{choi_injectivity_1977}
M.-D. Choi and E.~G. Effros.
\newblock Injectivity and operator spaces.
\newblock {\em Journal of Functional Analysis}, 24(2):156--209, Feb. 1977.

\bibitem{davidson_choquet_2013}
K.~R. Davidson and M.~Kennedy.
\newblock The {C}hoquet boundary of an operator system.
\newblock {\em {Duke J. Math.}}.
\newblock to appear. 

\bibitem{dritschel_boundary_2005}
M.~A. Dritschel and S.~A. McCullough.
\newblock Boundary representations for families of representations of operator
  algebras and spaces.
\newblock {\em Journal of Operator Theory}, 53(1):159{\textendash}167, 2005.

\bibitem{elliott_classification_1976}
G.~A. Elliott.
\newblock On the classification of inductive limits of sequences of semisimple
  finite-dimensional algebras.
\newblock {\em Journal of Algebra}, 38(1):29--44, Jan. 1976.

\bibitem{elliott_isomorphism_2013}
G.~A. Elliott, I.~Farah, V.~Paulsen, C.~Rosendal, A.~S. Toms, and
  A.~T{\"o}rnquist.
\newblock The isomorphism relation for separable {C}*-algebras.
\newblock {\em Mathematical Research Letters}, 20(6):1071--1080, 2013.

\bibitem{farah_model_?}
I.~Farah, B.~Hart, and D.~Sherman.
\newblock Model theory of operator algebras {II}: Model theory.
\newblock {\em Israel Journal of Mathematics},  201 (2014), no. 1, 477–505.

\bibitem{farah_descriptive_2012}
I.~Farah, A.~Toms, and A.~T{\"o}rnquist.
\newblock The descriptive set theory of {C}*-algebra invariants.
\newblock {\em International Mathematics Research Notices}, page rns206, Sept.
  2012.

\bibitem{farah_turbulence_2014}
I.~Farah, A.~S. Toms, and A.~T{\"o}rnquist.
\newblock Turbulence, orbit equivalence, and the classification of nuclear
  {C}*-algebras.
\newblock {\em Journal f{\"u}r die reine und angewandte Mathematik},
  688:101--146, Mar. 2014.

\bibitem{farenick_characterisations_2013}
D.~Farenick, A.~S. Kavruk, V.~I. Paulsen, and I.~G. Todorov.
\newblock Characterisations of the weak expectation property.
\newblock {\em {arXiv}:1307.1055}, July 2013.

\bibitem{farenick_operator_2012}
D.~Farenick and V.~I. Paulsen.
\newblock Operator system quotients of matrix algebras and their tensor
  products.
\newblock {\em Mathematica Scandinavica}, 111(2):210{\textendash}243, 2012.

\bibitem{ferenczi_complexity_2009}
V.~Ferenczi, A.~Louveau, and C.~Rosendal.
\newblock The complexity of classifying separable {B}anach spaces up to
  isomorphism.
\newblock {\em Journal of the London Mathematical Society}, 79(2):323--345,
  Apr. 2009.

\bibitem{glimm_certain_1960}
J.~G. Glimm.
\newblock On a certain class of operator algebras.
\newblock {\em Transactions of the American Mathematical Society},
  95(2):318--340, 1960.

\bibitem{goldbring_kirchberg_2014}
I.~Goldbring and T.~Sinclair.
\newblock On {K}irchberg's embedding problem.
\newblock {\em {arXiv}:1404.1861}, Apr. 2014.
\newblock {arXiv}: 1404.1861.

\bibitem{hamana_injective_1979-1}
M.~Hamana.
\newblock Injective envelopes of {C}*-algebras.
\newblock {\em Journal of the Mathematical Society of Japan}, 31(1):181--197,
  Jan. 1979.

\bibitem{hamana_injective_1979}
M.~Hamana.
\newblock Injective envelopes of operator systems.
\newblock {\em Publications of the Research Institute for Mathematical
  Sciences}, 15(3):773--785, 1979.

\bibitem{hjorth_classification_2000}
G.~Hjorth.
\newblock {\em Classification and Orbit Equivalence Relations}, volume~75 of
  {\em Mathematical Surveys and Monographs}.
\newblock American Mathematical Society, Providence, {RI}, 2000.

\bibitem{johnston_computing_2009}
N.~Johnston, D.~W. Kribs, and V.~I. Paulsen.
\newblock Computing stabilized norms for quantum operations via the theory of
  completely bounded maps.
\newblock {\em Quantum Information \& Computation}, 9(1-2):16{\textendash}35,
  2009.

\bibitem{johnston_minimal_2011}
N.~Johnston, D.~W. Kribs, V.~I. Paulsen, and R.~Pereira.
\newblock Minimal and maximal operator spaces and operator systems in
  entanglement theory.
\newblock {\em Journal of Functional Analysis}, 260(8):2407{\textendash}2423,
  2011.

\bibitem{kavruk_nuclearity_2011}
A.~S. Kavruk.
\newblock Nuclearity related properties in operator systems.
\newblock {\em  J. Operator Theory}, 71 (2014), no. 1, 95–156.

\bibitem{kechris_classical_1995}
A.~S. Kechris.
\newblock {\em Classical descriptive set theory}, volume 156 of {\em Graduate
  Texts in Mathematics}.
\newblock Springer-Verlag, New York, 1995.

\bibitem{latrach_facts_2005}
K.~Latrach, J.~M. Paoli, and P.~Simonnet.
\newblock Some facts from descriptive set theory concerning essential spectra
  and applications.
\newblock {\em Studia Mathematica}, 171(3):207--225, 2005.

\bibitem{paulsen_completely_2002}
V.~I. Paulsen.
\newblock {\em Completely bounded maps and operator algebras}, volume~78 of
  {\em Cambridge Studies in Advanced Mathematics}.
\newblock Cambridge University Press, Cambridge, 2002.

\bibitem{petersen_riemannian_2006}
P.~Petersen.
\newblock {\em Riemannian geometry}, volume 171 of {\em Graduate Texts in
  Mathematics}.
\newblock Springer, New York, second edition, 2006.

\bibitem{pisier_introduction_2003}
G.~Pisier.
\newblock {\em Introduction to operator space theory}, volume 294 of {\em
  London Mathematical Society Lecture Note Series}.
\newblock Cambridge University Press, Cambridge, 2003.

\bibitem{ruan_subspaces_1988}
Z.-j. Ruan.
\newblock Subspaces of {C}*-algebras.
\newblock {\em Journal of Functional Analysis}, 76(1):217--230, Jan. 1988.

\bibitem{sabok_completeness_2013}
M.~Sabok.
\newblock Completeness of the isomorphism problem for separable {C}*-algebras.
\newblock {\em {arXiv}:1306.1049}, June 2013.

\bibitem{sasyk_classification_2009}
R.~Sasyk and A.~T{\"o}rnquist.
\newblock The classification problem for von {N}eumann factors.
\newblock {\em Journal of Functional Analysis}, 256(8):2710--2724, Apr. 2009.

\bibitem{zielinski_complexity_2014}
J.~Zielinski.
\newblock The complexity of the homeomorphism relation between compact metric
  spaces.
\newblock {\em {arXiv}:1409.5523}, Sept. 2014.

\end{thebibliography}
\end{document}